\documentclass[11pt,twoside]{preprint}
\usepackage[a4paper,innermargin=1in,outermargin=1in,
bottom=1.5in,marginparwidth=1in,marginparsep
=3mm]{geometry}
\usepackage{amsmath,amsthm,amssymb,enumerate}
\usepackage{hyperref}
\usepackage[nameinlink,capitalize]{cleveref}
\usepackage{mathrsfs}
\usepackage{breakurl}
\usepackage[osf]{newtxtext}
\usepackage[varqu,varl]{inconsolata}
\usepackage[cal=boondoxo]{mathalfa}
\usepackage[bigdelims,vvarbb,cmintegrals]{newtxmath}

\usepackage[dvipsnames]{xcolor}
\usepackage{mhequ}
\usepackage{comment}
\usepackage{microtype}
\usepackage{pifont}

\colorlet{darkblue}{blue!90!black}
\colorlet{darkred}{red!90!black}
\colorlet{darkgreen}{green!60!black}

\usepackage{accents}

\newtheorem{theorem}{Theorem}[section]
\newtheorem{lemma}[theorem]{Lemma}

\newtheorem{corollary}[theorem]{Corollary}

\theoremstyle{definition}
\newtheorem{assumption}[theorem]{Assumption}
\newtheorem{definition}[theorem]{Definition}

\theoremstyle{remark}
\newtheorem{remark}[theorem]{Remark}

\AddToHook{env/lemma/begin}{\crefalias{theorem}{lemma}}
\AddToHook{env/proposition/begin}{\crefalias{theorem}{proposition}}
\AddToHook{env/corollary/begin}{\crefalias{theorem}{corollary}}
\AddToHook{env/assumption/begin}{\crefalias{theorem}{assumption}}
\AddToHook{env/definition/begin}{\crefalias{theorem}{definition}}
\AddToHook{env/example/begin}{\crefalias{theorem}{example}}
\AddToHook{env/remark/begin}{\crefalias{theorem}{remark}}

\crefname{lemma}{Lemma}{Lemmata}

\def\scal#1{\langle #1 \rangle}

\newcommand{\vn}[1]{{\vert\kern-0.23ex\vert\kern-0.23ex\vert #1 
    \vert\kern-0.23ex\vert\kern-0.23ex\vert}}
\usepackage{mathtools}

\allowdisplaybreaks

\newcommand{\D}{\partial}

\def\eps{\varepsilon}

\newcommand{\R}{{\mathbb{R}}}

\renewcommand{\P}{\mathbb{P}}
\def\T{\mathbb{T}}

\newcommand{\E}{\mathbb{E}}
\newcommand{\Z}{\mathbb{Z}}

\newcommand{\N}{\mathbb{N}}

\newcommand{\cC}{\mathcal{C}}

\def\path{\mathrm{path}}

\author{Konstantinos Dareiotis\thanks{School of Mathematics, University of Leeds, Leeds, U.K., \url{K.Dareiotis@leeds.ac.uk}}\,\,
and M\'at\'e Gerencs\'er\thanks{Institute of Analysis and Scientific Computing, TU Wien, Austria, \url{mate.gerencser@asc.tuwien.ac.at}}\,}
\title{It\^o perspective on variance renormalisation}
\begin{document}

\maketitle

\begin{abstract}
We show that the It\^o solutions of the nonlinear stochastic heat equation
$$
 \D_t u^\eps- \Delta u^\eps =\eps^{3/4} g (u^\eps) \nabla \xi_\eps,
$$
where $ \xi_\eps$ denotes the mollification in space at scale $\eps>0$ of a space-time white noise $\xi$,
converge in law, as $\eps\to 0$, to the solution of the stochastic heat equation with right-hand side $cg'g(u)\xi$ with a constant $c>0$.
Since the noise $\nabla\xi$ is supercritical, the small prefactor is not unexpected to obtain a limit, but the exponent $3/4$ is not predicted by naive scaling arguments.
The case $g(u)=u$, modulo a Cole–Hopf transform, corresponds to the result of \cite{H25} for the KPZ equation. Our argument is relatively short and relies solely on stochastic analytic techniques.
\end{abstract}

\tableofcontents

\section{Introduction}
In recent years, there has been some interest in stochastic PDEs driven by noise which is rougher than white (see \cite{long,Mate,GT,H25,LiRen}). 
Consider, formaly, the nonlinear stochastic PDE of the form:
\begin{equ}\label{eq:intro}
    \D_t u= \Delta u +  g (u) (-\Delta)^{\gamma/2} \xi 
\end{equ}
in $1+1$ dimensions, where $\gamma>0$ and $\xi$ is a space-time white noise. If $u$ satisfies \eqref{eq:intro}, then 
 the rescaled (under parabolic scaling ) solution $u^{(\lambda)}_t(x):=u_{\lambda^2 t}(\lambda x)$ satisfies 
\begin{equ}
   \D_t u^{(\lambda)}=\Delta u^{(\lambda)}+\lambda^{1/2-\gamma}g(u^{(\lambda)})(-\Delta)^{\gamma/2}  \xi^{(\lambda)},
\end{equ}
where $\xi^{(\lambda)}(t,x):=\lambda^{3/2}\xi(\lambda^2t,\lambda x)$ is a space-time white noise . For any $\gamma < 1/2$ the nonlinearity vanishes on small scales, and so the equation is scaling subcritical. 
This heuristic however gives an incorrect prediction for solvability of the equation:
the solution theory of \eqref{eq:intro} only extends up to $\gamma<1/4$ \cite{long, Mate}. This barrier $1/4$ is due to a phenomenon of variance blow-up. To make things concrete, let us fix the spatial domain to be the $1$-dimensional torus $\T=\R/\Z$, let $\xi_{\eps}$ be a smooth in space and time approximation of $\xi$ on parabolic scale $\eps$, and let $X_{\eps}$ be the solution of the linear equation $\D_t X_{\eps}=\Delta X_{\eps}+(-\Delta)^{\gamma/2}\xi _{\eps}$ starting from zero, that is,
$$
X_{\eps}(t,x) := \int_0^t \int_\T p_{t-s}(x-y) (-\Delta)^{\gamma/2}\xi _{\eps}(ds,dy).
$$
Then for any $\gamma\geq 1/4$ and any nontrivial test function $\varphi\in C_c^\infty(\R_+\times\T)$, the variance of  $\scal{ X_{\eps} (-\Delta)^{\gamma/2}\xi _{\eps},\varphi}$ diverges as $\eps\to 0$. Therefore, already the first nontrivial step of Picard's iteration for the mild formulation of the equation can not be performed.
Such a divergence can not be removed by an additive renormalising counterterm.
The blow-up of the variance occurs also if one takes an approximation $\xi_\eps$ only in the spatial variable and interprets the product through an It\^o integral
\begin{equ}
\scal{X_\eps(-\Delta)^{\gamma/2}\xi_\eps,\varphi}:=\int_\R\int_\T\varphi(t,x)\int_0^t \int_\T p_{t-s}(x-y)(-\Delta)^{\gamma/2}\xi_\eps(ds,dy)(-\Delta)^{\gamma/2}\xi_\eps(dt,dx).
\end{equ}

The case $g(u)=u$ is of particular interest due to its relation to the KPZ equation via the Cole-Hopf transform: formally $h=\log u$ solves
\begin{equ}\label{eq:KPZ}
    \partial_t h=\Delta h+(\nabla h)^2+ (-\Delta)^{\gamma/2} \xi .
\end{equ}
When $\gamma=1$ (recall that $\nabla\xi\overset{\mathrm{law}}{=}(-\Delta)^{1/2}\xi$), \cite{H25} showed how the variance blow-up can be tamed by a \emph{multiplicative} renormalisation: 
the solutions $h^\eps$ of the equations
\begin{equ}\label{eq:KPZ-var}
    \partial_t h^\eps=\Delta h^\eps+(\nabla h^\eps)^2-C_\eps+\eps^{3/4}\nabla\xi_\eps,
\end{equ}
with an appropriate $C_\eps\sim \eps^{-1}$,
converge in law to the solution of 
\begin{equ}\label{eq:KPZ-var}
    \partial_t h=\Delta h+(\nabla h)^2+c\xi,
\end{equ}
 with some constant $c>0$. 
The same conclusion is expected to hold when one replaces $\eps^{3/4}\nabla\xi_\eps$ by $\eps^{\gamma-1/4}(-\Delta)^{\gamma/2}\xi_\eps$,
for any $\gamma>1/4$ (this is proved for a variant of the problem \cite{GT}).

The KPZ equation has the special advantage that the so-called Da Prato-Debussche trick is applicable, which isolates the most singular non-vanishing object 
$\eps^{3/2}(\nabla P\ast\xi_\eps)^2$. For each $\eps>0$ this random field belongs to the second Wiener chaos, but it converges in law to a multiple of a white noise.
In the language of the tree-notation, thanks to  this trick every new noise appears at the leaves of  trees.
This greatly simplifies both the algebra and the probabilistic estimates for \eqref{eq:KPZ-var}.
Even with this simplification, the proof of \cite{H25} builds on the theory of regularity structures \cite{H0}, therefore the overall input to the proof is rather lengthy.
In equations where no such trick is available, the regularity structure machinery even with much less singular noise is highly demanding, see \cite{MY} treating the $2$-dimensional generalised parabolic Anderson model.
It is currently out of reach to treat the $3$-dimensional version of \cite{MY}, which scaling-wise corresponds to the $1+1$-dimensional space-time white noise-driven nonlinear multiplicative stochastic heat equation treated in the present paper.
The main difference is that the latter is amenable to It\^o calculus, which our proof leverages.

Our main result reads as follows. Denote $\xi_\eps:=P_{\eps^2} \xi$, where $(P_t)_{t\geq 0}$ is the heat semigroup on $\T$ acting only in the spatial variable.
  \begin{theorem}\label{thm:main}
      Let $\psi:\T\to\R$ be $1/4$-H\"older continuous and let $g:\R\to\R$ have bounded derivatives of order $1$, $2$, and $3$.
      For $\eps>0$, let $u^\eps$ denote the It\^o solution of 
      \begin{equs}      \label{eq:main_equation_eps}
    \D_t u^\eps= \Delta u^\eps + \eps^{3/4} g (u^\eps) \nabla \xi_\eps, \qquad u|_{t=0}= \psi
\end{equs}
on $[0,1]\times\T$,
and let $u$ denote the It\^o solution of
\begin{equs}      \label{eq:main_equation_limit}
    \D_t u= \Delta u +  \frac{1}{8\sqrt[\leftroot{2}\uproot{2}4]{\pi}
}g'g (u)\xi, \qquad u|_{t=0}= \psi
\end{equs}
on $[0,1]\times\T$.
Then, as $\eps \to 0$,  $u^\eps$  converge weakly to $u$, as random variables with values in   $C([0,1] \times \mathbb{T})$. 
  \end{theorem}

  \begin{remark}
      By the It\^o solution we mean a probabilistically strong, analytically mild solution, see Definition \ref{def:soln} below.
  \end{remark}

Taking the Cole-Hopf transform of the result of
\cite{H25}, the present paper with $g(u)=u$ gives a new and simpler proof in the case of spatial mollification, and generalise to the case of nonlinear $g$.
The It\^o structure of the equation allows for a rather down-to-earth argument: the only estimates used are standard analytic inequalities (Cauchy-Schwarz, Minkowski), stochastic analytic inequalities (BDG), heat kernel estimates, and the stochastic sewing lemma \cite{SSL}.

\begin{remark}
    Although the problem is more singular than the fractional Brownian-driven SDE example in \cite{H25} and the $2$-dimensional gPAM studied in \cite{MY} (in both cases the limiting equation is only ``borderline'' ill-posed in the classical sense), the change of nonlinearity $g\leadsto cgg'$ is the same.
\end{remark}

\begin{remark} 
   As in some other examples \cite{H25, GT, MY}, the limiting noise plays a similar role in the equation as in the prelimit  (i.e. a multiplicative noise, albeit with a different nonlinearity).
   However, there are also examples of variance renormalisation where the new noise appears at a different part of the equation:
in \cite{BBMvar} it is shown that a supercritical random initial condition for the Benjamin-Bona-Mahony equation,
appropriately renormalised, produces a random forcing in the limit.
\end{remark}

\begin{remark}
Since the aforementioned diverging variance is of order $\eps^{-3/2}$, the  prefactor $\eps^{3/4}$ ensures that the first nontrivial Picard iterate stays bounded in $L^2$.
    One can also heuristically guess $\eps^{3/4}$ to be the correct prefactor to take since it makes, in a certain norm, the noise to be of the same magnitude as $(-\Delta)^{1/4}\xi_\eps$ which is borderline for the variance blow-up: setting
    \begin{equ}
        \|\zeta\|_{[\alpha]}:=\sup_{s\in[0,1]}\sup_{n\in\N} \frac{\|\zeta(\mathbf{1}_{[s,s+n^{-2}]}\otimes e_n)\|_{L_2(\Omega)}}{n^{-\alpha}},
    \end{equ}
    where $(e_n)_{n\in\N}$ is the usual eigenbasis of the Laplacian,
    one has that both $\|(-\Delta)^{1/4}\xi_\eps\|_{[-7/4]}$ and $\|\eps^{3/4}\nabla\xi_\eps\|_{[-7/4]}$ are uniformly bounded in $\eps\in(0,1]$ and neither go to $0$ as $\eps\to 0$.
\end{remark}

The rest of the paper is structured as follows. We continue the introduction by setting up the notation for the rest of the paper in Section \ref{sec:notation}. In Section \ref{sec:apriori} we establish a uniform in $\eps$ bound on $u^\eps$, implying tightness of their laws.
In Section \ref{sec:limiting} we identify the law of any subsequential limit as that of the solution of \eqref{eq:main_equation_limit}. In the passage to the limit the characterisation of stochastic integrals via the stochastic sewing lemma turns out to be rather convenient.

\subsection{Setup and notation}\label{sec:notation}
Set $\mathbb{N}_0:= \mathbb{N}\cup \{0\}$. Set $[0,1]_\leq^2=\{(s,t)\in[0,1]^2:s\leq t\}$
and $[0,1]_\leq^3=\{(s,v,t)\in[0,1]^2:s\leq v\leq t\}$.
We consider a filtered probability space $(\Omega, \mathcal{F}, \mathbb{F}=(\mathcal{F}_t)_{t \in [0,1]}, \P)$ carrying a sequence $(w^n)_{n=1}^\infty$ of independent standard $\mathbb{F}$-Brownian motions on $[0,1]$. The predictable $\sigma$-algebra on $\Omega \times [0,1]$ will be denoted by $\mathcal{P}$. A $\mathcal{P} \otimes \mathcal{B}(\mathbb{T})$-measurable map $u : \Omega \times [0,1] \times \mathbb{T} \to \mathbb{R}$ will be called predictable random field. We denote by $\mathscr{L}_2$ the collection of all predictable random fields $u$, such that 
\begin{equs}
    \E \int_0^1 \int_\mathbb{T} |u_t(x)|^2  \, dx dt < \infty . 
\end{equs}
Let $p \in [1, \infty)$ and  $\alpha \in (0,1]$. By $L_p$ we denote the usual $L_p(\Omega, \mathcal{F}, 
\mathbb{P})$ space and by $\|\cdot\|_{L_p}$ we mean the corresponding norm.
Classical H\"older spaces are defined via the norms
\begin{equ}
   \|f\|_{C^\alpha(\T)}:=\sup_{x\in \T} |f(x)|+ \sup_{x\neq y\in \T} \frac{|\varphi(x)-\varphi(y)|}{|x-y|^\alpha}.
\end{equ}
For a measurable map $\varphi: \Omega \times \mathbb{T} \to \R$, we set 
\begin{equs}
 \|\varphi\|_{\mathcal{C}^0_p} := \sup_{x\in \mathbb{T}} \|\varphi(x)\|_{L_p},
\qquad  [\varphi]_{\mathcal{C}^\alpha_p}:= \sup_{x\neq y\in\T} \frac{\|\varphi(x)-\varphi(y)\|_{L_p}}{|x-y|^\alpha},
\end{equs}
and
\begin{equs}
\|\varphi\|_{\mathcal{C}^\alpha_p}=\|\varphi\|_{\mathcal{C}^0_p}+[\varphi]_{\mathcal{C}^\alpha_p}.
\end{equs}
Similarly, for a measurable map $u: \Omega \times[0,1] \times \mathbb{T}\to \R$, we set 
\begin{equs}
 \|u\|_{\mathcal{C}^{0,0}_p([s,t])} := \sup_{r \in [s,t]}\|u_r\|_{\mathcal{C}^0_p},
 \qquad    [u]_{\mathcal{C}^{0,\alpha}_p([s,t])} &:= \sup_{r \in [s,t]} [u_r]_{\mathcal{C}^\alpha_p}, 
\end{equs}
and 
\begin{equs}
      \|u\|_{\mathcal{C}^{0,\alpha}_p([s,t])}:= \|u\|_{\mathcal{C}^{0,0}_p([s,t])}+ [u]_{\mathcal{C}^{0,\alpha}_p([s,t])}.
\end{equs}
In other words, $\cC^\alpha_p$ is really a shorthand for $C^\alpha(\T; L_p)$.
 Let $(e_n)_{n=1}^\infty$ be the orthonormal basis of $L_2(\mathbb{T})$ given by 
\begin{equs}
    e_1(x)\equiv 1, \qquad e_n(x)
=
\begin{cases}
    \sqrt{2} \sin ( k 2\pi x), \qquad  &\text{if }n=2k
    \\
    \sqrt{2} \cos ( k 2\pi x), \qquad  &\text{if }n=2k+1,
\end{cases}
\,\text{for $n \geq 2$}.
\end{equs}
Note that with a universal constant $C$,  for any $n \in \mathbb{N}$ and any $\alpha \in [0,1]$, one has
\begin{equs} \label{eq:growth_basis}
    \| e_n\|_{C^\alpha} \leq C n^{\alpha}. 
\end{equs}
The space-time noise on $[0,1] \times \mathbb{T}$ and its spatial mollification are given by 
\begin{equs}
    \xi = \sum_{n=1}^\infty e_n \D_t w^n,\qquad \xi_{\eps}= \sum_{n=1}^\infty (P_{\eps^2}e_n) \D_t w^n,
\end{equs}
for $\eps>0$, 
where $(P_t)_{t\in [0,1]}$ denotes the heat semigroup acting in the spatial variable, that is 
\begin{equs}
    P_tf(x) = \int_{\mathbb{T}} p_t(x-y) f(y) \, dy, \qquad p_t(x)=  \frac{1}{\sqrt{4\pi t}} \sum_{k \in \mathbb{Z}} \exp \Big( \frac{(x-k)^2}{4t} \Big). 
\end{equs}
We also use the notation $P_t f$ for space-time functions, still acting only in the spatial variable, for example $\xi_\eps=P_{\eps^2}\xi$.
For $f \in \mathcal{L}_2$, the stochastic integral of $f$ with respect to the white noise $\xi$ is well defined (see, e.g., \cite{Walsh}). Furthermore, one has the identity
\begin{equs}    \label{eq:Walsh_vs_Krylov}
    \int_0^t \int_{\mathbb{T}} f_s(x) \,  \xi(dx, ds)= \sum_{n=1}^\infty \int_0^t (f_s, e_n) \, dw^n_s, \qquad t \in [0,1],
\end{equs} 
where $(\cdot, \cdot)$ denotes the inner product in $L_2(\mathbb{T})$. 
In addition, it is easy to see that if $f \in \mathcal{L}_2$, then $\nabla P_{\eps^2}f \in \mathcal{L}_2$, the integral of $f$ with respect to $\nabla \xi_{\eps}$ is well defined,  and it satisfies 
\begin{equs}
    \int_0^t \int_{\mathbb{T}} f_r(x) \, \nabla \xi_\eps(dx,dr)= \int_0^t \int_{\mathbb{T}} \nabla P_{\eps^2}f_r(z) \,  \xi(dz, dr) .      \label{eq:integration_white_noise}
\end{equs}
Finally, the following shorthand for stochastic integrals will be  used throughout the paper
\begin{equs}[eq:K-J-notation]
    K_{s,t}^\eps[f]& := \eps^{3/4}\int_s^t \int_{\mathbb{T}} f_r(x) \, \nabla \xi_\eps(dx,dr),
    \\
    J^\eps_{s,t}[f](x)&:= \eps^{3/4}\int_s^t \int_{\T}p_{t-r}(x-y) f_r(y) \nabla \xi_\eps(dy, dr), \qquad x \in \mathbb{T},
   \end{equs}
   for any $f$ such that the corresponding integrals above are meaningful. 
   Often the function $f$ has more than one time parameter, in which case the running variable (in the above formulae, $r$) is denoted by $\cdot$. Occasionally such integrals appear in a nested way, the two running variables are then distinguished by using $\cdot$ and $\diamond$. As an example,
   \begin{equ}
       K^{\eps}_{s,t}\big[J_{s,\cdot}^\eps[P_{\diamond-s}f]\big]=\eps^{3/2}\int_s^t\int_\T\int_s^r\int_\T p_{r-\tau}(x-y)P_{\tau-s}f_s(y)
       \nabla\xi_\eps(dy,d\tau)
       \nabla\xi_\eps(dx,dr).
   \end{equ}

The following assumption encompasses the main conditions of Theorem \ref{thm:main}.
\begin{assumption}       \label{as:g}
    One has $\psi\in C^{1/4}(\T)$ and the map $g: \R \to \R$ is three time continuously differentiable and there exists a constant $C_g$ such that for all $r \in \mathbb{R}$ we have 
    \begin{equs}
      |g(0)|+  |g'(r)|+|g''(r)|+|g'''(r)|  \leq C_g. 
    \end{equs}
\end{assumption}

\begin{definition}\label{def:soln}
    A predictable random field $u^\eps$ is called solution of \eqref{eq:main_equation_eps} if $\|u^\eps\|_{\mathcal{C^{0,0}_2}([0,1])} < \infty $, it is continuous in $(t,x)$, and for all $(t,x) \in [0,1] \times \mathbb{T}$, with probability one we have
  \begin{equs}
      u^{\eps}_t(x) = P_t\psi(x)+ \eps^{-3/4}\int_0^t \int_{\T}p_{t-r}(x-y) g(u^{\eps}_r(y))  \nabla \xi_\eps(dy, dr).
  \end{equs}
Similarly,  a predictable random field $u$ is called solution of \eqref{eq:main_equation_limit} if $\|u\|_{\mathcal{C^{0,0}_2}([0,1])} < \infty $, it is continuous in $(t,x)$, and for all $(t,x) \in [0,1] \times \mathbb{T}$, with probability one we have
  \begin{equs}
      u_t(x) = P_t\psi(x)+ \frac{1}{4\sqrt[\leftroot{2}\uproot{2}4]{\pi}}\int_0^t \int_{\T}p_{t-r}(x-y) g'g(u_r(y))  \xi (dy, dr).
  \end{equs}
\end{definition}

\begin{remark}
    It is fairly standard that Assumption \ref{as:g} is more than enough to guarantee that both \eqref{eq:main_equation_eps} for any $\eps>0$ and \eqref{eq:main_equation_limit} have a unique solution. For the latter one has to be slightly careful, since $g'g$ is only locally Lipschitz continuous. But since it is globally of linear growth, one can easily reduce to the globally Lipschitz case by a stopping time argument, see e.g. \cite{Martabook}.
\end{remark}

In proofs, the notation $a\lesssim b$ abbreviates the existence of $C>0$ such that $a\leq C b$, such that moreover $C$ depends only on the parameters claimed in the corresponding statement.

\section{A priori estimates}\label{sec:apriori}

\subsection{Estimate for the solutions}\label{sec:apriori-1}
The main purpose of this section is to prove the following estimate on $u^\eps$.

\begin{theorem}         \label{lem:niform_bound_1/4}
    Let Assumption \ref{as:g} hold.
    For all $p \in [1, \infty)$, there exists a constant $N=N(p,C_g)$ such that for all $\eps>0$ we have 
    \begin{equs}
        \| u^\eps\|_{\mathcal{C}^{0,1/4}_p([0,1])} \leq N (1+\|\psi\|_{C^{1/4}(\T)}).
    \end{equs}
\end{theorem}
The exact choice of the exponent $1/4$ is crucial, as is the fact that in the $\mathcal{C}^{0,1/4}_p([0,1])$ norm the suprema in time and space are outside of the $L_p$ norm in $\omega$ (exchanging them via Kolmogorov continuity would also cost in the regularity exponent).

We will repeatedly use the following standard heat kernel estimates. For any $\alpha \in[0,1]$, $\beta>-1$, and  $k \in \mathbb{N}_0$, there exists a constant $N$ such that for all $t \in (0,1]$ we have 
    \begin{equs}
  \| \nabla^k p_t\|_{L_2(\mathbb{T})} & \leq N t^{-k/2-1/4},  \label{eq:HK_0}
    \\
    \int_\T |x|^\beta |\nabla^k p_t(x)|dx&\leq N t^{(\beta-k)/2}, \label{eq:HK-power}\\
     \| \nabla^k p_t(x-\cdot)- \nabla^k p_t(x'-\cdot)\|_{L_2(\mathbb{T})} &\leq N |x-x'|^\alpha t^{-(\alpha+k)/2-1/4}.    \label{eq:HK_alpha}
    \end{equs}
Variants of the heat kernel estimates in the $\cC^\alpha_p$ spaces will also be used. Given $\alpha, \gamma \in [0,1]$ with $\gamma \geq \alpha$ and $k \in \mathbb{N}$, there exist a constant $N$ such that for all $p\in[1,\infty)$, all $\varphi \in \mathcal{C}^{\alpha}_p$ and all $t \in (0,1]$
we have 
\begin{equs}
    \|\nabla^k P_t\varphi\|_{\mathcal{C}^0_p} \leq N_1 t^{(\alpha-k)/2}\|\varphi\|_{\mathcal{C}^\alpha_p}         \label{eq:HK_omega_deriv}
    \\
      \|P_t\varphi\|_{\mathcal{C}^\gamma_p} \leq N_2 t^{(\alpha-\gamma)/2}\|\varphi\|_{\mathcal{C}^\alpha_p}.   \label{eq:HK_omega_holder}
\end{equs}
These follow from Minkowski's inequality (for the norm $\|\cdot\|_{L_p}$) and standard heat kernel properties.

\begin{lemma}     \label{lem:J_in_Lp}
   Let $p \in [2, \infty)$, $\alpha \in [0, 1/4]$.  There exists a constant $N=N(p, \alpha)$ such that for all $\eps>0$, all $(s,t)\in [0,1]^2_{\leq}$, and all $f\in\mathcal{C}^{0,\alpha}_p([s,t])$, the following bound holds:
   \begin{equs}
         \|J_{s,t}^\eps[f]\|_{\mathcal{C}^\alpha_p}^2&  \leq N  \eps^{3/2} \int_s^t (\eps^2+t-r)^{-\alpha-3/2}  \|f\|_{\mathcal{C}^{0,0}_p([s,r])}^2  \,  dr
      \\
    & \qquad \qquad + N  \int_s^t   (t-r)^{-\alpha-{1/2}}[f]_{\mathcal{C}^{0,1/4}_p([s,r])}^2 \, dr.   \label{eq:main_estimate_J}
   \end{equs}
\end{lemma}
\begin{proof}
   For arbitrary $x,x' \in \mathbb{T}$, by the definition of $J^\eps_{s,t}[f]$ in  \eqref{eq:K-J-notation}, the identity \eqref{eq:integration_white_noise}, and the  Burkholder-Davis-Gundy (BDG) inequality,   we have
    \begin{equs}
      \|J_{s,t}^\eps&[f](x)-J_{s,t}^\eps[f](x')\|_{L_p}^2
        \\
        & \lesssim  \eps^{3/2} \Big\| \int_s^t \int_{\mathbb{T}}\Big( \int_{\mathbb{T}}G_{t-r}(x,x',y)f_r(y) \nabla p_{\eps^2}(y-z)\, dy \Big)^2 \, dz dr \Big\|_{L_{p/2}},             \label{eq:ingredient_1}
    \end{equs}
    where
    \begin{equs}
        G_{t-r}(x,x',y)=p_{t-r}(x-y)-p_{t-r}(x'-y).
    \end{equs}
By smuggling in $f_r(z)$ we can write
    \begin{equs}
        \Big\| \int_s^t \int_{\mathbb{T}}\Big( \int_{\mathbb{T}}G_{t-r}(x,x',y) f_r(y) \nabla p_{\eps^2}(y-z)\, dy \Big)^2 \, dz dr \Big\|_{L_{p/2}}\lesssim A_1+A_2,       \label{eq:ingredient_2}
    \end{equs}
    where
    \begin{equs}
        A_1&:=    \Big\| \int_s^t \int_{\mathbb{T}}\Big( \int_{\mathbb{T}}G_{t-r}(x,x',y) f_r(z) \nabla p_{\eps^2}(y-z)\, dy \Big)^2 \, dz dr \Big\|_{L_{p/2}},
        \\
        A_2&:=    \Big\| \int_s^t \int_{\mathbb{T}}\Big( \int_{\mathbb{T}}G_{t-r}(x,x',y) \big(f_r(y)-f_r(z)\big) \nabla p_{\eps^2}(y-z)\, dy \Big)^2 \, dz dr \Big\|_{L_{p/2}}.
    \end{equs}
    We start with an estimate for $A_1$. By integration by parts and the semigroup property we have the identity
    \begin{equs}
 \Big(\int_{\mathbb{T}}G_{t-r}(x,x',y) f_r(z) \nabla p_{\eps^2}(y-z)\, dy\Big)^2
       = |\nabla_z G_{\eps^2+t-r}(x,x',z)|^2  |f_r(z)|^2.
    \end{equs}
    Using the above and Minkowski's inequality we get 
    \begin{equs}
         A_1 &\lesssim  \int_s^t \int_{\mathbb{T}}|\nabla_zG_{\eps^2+t-r}(x,x',z)|^2  \|f_r(z)\|_{L_p}^2  \,dz dr
         \\
         & \lesssim  \int_s^t \|\nabla_z G_{\eps^2+t-r}(x,x',\cdot)\|_{L_2(\mathbb{T})}^2  \|f\|_{\mathcal{C}^{0,0}_p([s,r])}^2  \, dr
         \\
           & \stackrel{\eqref{eq:HK_alpha}}{\lesssim}  |x-x'|^{2\alpha} \int_s^t (\eps^2+t-r)^{-\alpha-3/2}  \|f\|_{\mathcal{C}^{0,0}_p([s,r])}^2  \,  dr.
    \end{equs}
As for $A_2$, we use the fact that 
    $$
    \nabla p_{\eps^2}(y-z)=\eps^{-2}(y-z)p^{1/2}_{\eps^2}(y-z)p^{1/2}_{\eps^2}(y-z).
    $$ 
Therefore, by the Cauchy-Schwarz inequality we get   
    \begin{equs}
       \Big( \int_{\mathbb{T}}&G_{t-r}(x,x',y) \big(f_r(y)-f_r(z) \big) \nabla p_{\eps^2}(y-z)\, dy \Big)^2
       \\
       & \lesssim \eps^{-4} \int_{\mathbb{T}}|G_{t-r}(x,x',y)|^2 p_{\eps^2}(y-z)\, dy 
          \int_{\mathbb{T}}|f_r(y)-f_r(z)|^2 |y-z|^2p_{\eps^2}(y-z)\, dy
          \\
          & =\eps^{-4} P_{\eps^2}|G_{t-r}(x,x',\cdot)|^2(z) \int_{\mathbb{T}}|f_r(y)-f_r(z)|^2 |y-z|^2p_{\eps^2}(y-z)\, dy.
    \end{equs}
  By Minkowski's inequality we get
    \begin{equs}
        A_2 & \lesssim  \eps^{-4}\int_s^t  \int_{\mathbb{T}} P_{\eps^2}|G_{t-r}(x,x',\cdot)|^2(z)  \, \int_{\mathbb{T}}\|f_r(y)-f_r(z)\|_{L_p}^2 |y-z|^2p_{\eps^2}(y-z)\, dy  dz dr.     
        \\
        & \lesssim  \eps^{-4}\int_s^t  \int_{\mathbb{T}} P_{\eps^2}|G_{t-r}(x,x',\cdot)|^2(z) \,  [f]_{\mathcal{C}^{0,1/4}_p([s,r])}^2\, \int_{\mathbb{T}}|y-z|^{5/2} p_{\eps^2}(y-z)\, dy  dz dr
        \\
         & \stackrel{\eqref{eq:HK-power}}{\lesssim}  \eps^{-3/2}\int_s^t   \int_{\mathbb{T}} P_{\eps^2}|G_{t-r}(x,x',\cdot)|^2(z)  \, dz \,  [f]_{\mathcal{C}^{0,1/4}_p([s,r])}^2  \, dr
         \\
          & \lesssim  \eps^{-3/2}\int_s^t  \|G_{t-r}(x,x',\cdot)\|_{L_2(\mathbb{T})}^2   [f]_{\mathcal{C}^{0,1/4}_p([s,r])}^2  \, dr
          \\
            & \stackrel{\eqref{eq:HK_alpha}}{\lesssim}  \eps^{-3/2}|x-x'|^{2\alpha}\int_s^t   (t-r)^{-\alpha-{1/2}}[f]_{\mathcal{C}^{0,1/4}_p([s,r])}^2  \, dr.
    \end{equs}  
    Hence, by \eqref{eq:ingredient_1}, \eqref{eq:ingredient_2}, and the estimates on $A_1, A_2$, we arrive at 
    \begin{equs}
      \|J_{s,t}^\eps[f](x)-J_{s,t}^\eps[f](x')\|_{L_p}^2 & \lesssim \eps^{3/2}  |x-x'|^{2\alpha} \int_s^t (\eps^2+t-r)^{-\alpha-3/2}  \|f\|_{\mathcal{C}^{0,0}_p([s,r])}^2  \,  dr
    \\
      &  \qquad + |x-x'|^{2\alpha}\int_s^t   (t-r)^{-\alpha-{1/2}}[f]_{\mathcal{C}^{0,1/4}_p([s,r])}^2  \, dr,
    \end{equs}
    which in turn gives
    \begin{equs}
     \,    [J_{s,t}^\eps[f]]_{\mathcal{C}^\alpha_p}^2 & \lesssim \eps^{3/2} \int_s^t (\eps^2+t-r)^{-\alpha-3/2}  \|f\|_{\mathcal{C}^{0,0}_p([s,r])}^2  \,  dr
      \\
    & \qquad \qquad + \int_s^t   (t-r)^{-\alpha-{1/2}}[f]_{\mathcal{C}^{0,1/4}_p([s,r])}^2  \, dr.  \label{eq:alpha_geq_0}
    \end{equs}
    One can repeat the proof by replacing $G_{t-r}(x,x',y)$ with $p_{t-r}(x,y)$ and repeating the above estimates with $\alpha=0$ to get 
  \begin{equ}    \|J_{s,t}^\eps[f]\|_{\mathcal{C}^0_p}^2  \lesssim   \eps^{3/2} \int_s^t (\eps^2+t-r)^{-3/2}  \|f\|_{\mathcal{C}^{0,0}_p([s,r])}^2  \,  dr
 + \int_s^t   (t-r)^{-{1/2}}[f]_{\mathcal{C}^{0,1/4}_p([s,r])}^2  \, dr,      \label{eq:alpha_equal_0}
    \end{equ}
    which is exactly \eqref{eq:main_estimate_J} for $\alpha=0$. 
    Finally, summing 
    \eqref{eq:alpha_geq_0} and \eqref{eq:alpha_equal_0} gives \eqref{eq:main_estimate_J} for $\alpha \in (0,1/4]$.
    This finishes the proof. 
\end{proof}

\begin{corollary}
   Let $p \in [2, \infty)$.  There exists a constant $N=N(p)$ such that for all $\eps>0$, all $(s,t)\in [0,1]^2_{\leq}$, and all $f\in\cC_p^{0,1/4}([s,t])$, the following bounds hold:
   \begin{equs}         \|J_{s,t}^\eps[f]\|_{\mathcal{C}^0_p}^2&  \leq N   \int_s^t (t-r)^{-3/4}  \|f\|_{\mathcal{C}^{0,0}_p([s,r])}^2  \,  dr
    +N  \int_s^t   (t-r)^{-{1/2}}[f]_{\mathcal{C}^{0,1/4}_p([s,r])}^2 \, dr,        \label{eq:J-C0}
\\
      \,   \|J_{s,t}^\eps[f]\|_{\mathcal{C}^{1/4}_p}^2 &  \leq  N\|f\|_{\mathcal{C}^{0,0}_p([s,t])}^2 + N \int_s^t (t-r)^{-3/4} [f]_{\mathcal{C}^{0, 1/4}_p([s,r])}^2\, dr. \label{eq:J-C1/4}
   \end{equs}
 
\end{corollary}
\begin{proof}
    The  bound \eqref{eq:J-C0} follows by applying Lemma \ref{lem:J_in_Lp} with $\alpha=0$ and using the trivial inequality $(\eps^2+t-r)^{-3/2} \leq \eps^{-3/2}(t-r)^{-3/4}$.   
    The bound \eqref{eq:J-C1/4} follows by applying Lemma \ref{lem:J_in_Lp} with $\alpha=1/4$ and noticing that 
    \begin{equs}
         \int_s^t (\eps^2+t-r)^{-7/4}  \|f\|_{\mathcal{C}^{0,0}_p([s,r])}^2  \,  dr\lesssim  \int_s^t (\eps^2+t-r)^{-7/4}\, dr \, \|f\|_{\mathcal{C}^{0,0}_p([s,t])}^2  \lesssim  \eps^{-3/2} \|f\|_{\mathcal{C}^{0,0}_p([s,t])}^2.
    \end{equs}
\end{proof}
We can now prove the main a priori bound.
    \begin{proof}[Proof of Theorem \ref{lem:niform_bound_1/4}]
     Notice that $\|u^\eps\|_{\mathcal{C}^{0, 1/4}_p([0,1])}< \infty$, since the noise $\xi_\eps$ is smooth in space.
    Using the form $u^\eps_t=P_t\psi + J^\eps_{0, t}[g(u^\eps)]$ and the bound \eqref{eq:J-C1/4}, we get that  for all $t \in [0,1]$ 
    \begin{equs}
      \|u^\eps_t\|_{\mathcal{C}^{1/4}_p}^2  \lesssim  \|\psi\|_{C^{1/4}}^2+   \|g(u^\eps)\|_{\mathcal{C}^{0,0}_p([0,t])}^2 + \int_0^t (t-r)^{-3/4} [g(u^\eps)]_{\mathcal{C}^{0, 1/4}_p([0,r])}^2\, dr.
    \end{equs}
    Further, using the fact that $g$ is Lipschitz continuous, we get 
    \begin{equs}
        \|u^\eps_t\|_{\mathcal{C}^{1/4}_p}^2 \lesssim  1+ \|\psi\|_{C^{1/4}}^2+   \|u^\eps\|_{\mathcal{C}^{0,0}_p([0,t])}^2 + \int_0^t (t-r)^{-3/4} \|u^\eps\|_{\mathcal{C}^{0, 1/4}_p([0,r])}^2\, dr.
    \end{equs}
  We can apply an appropriate version of the Henry-Gronwall inequality, see Lemma \ref{lem:Gronwal}, to obtain that  that for any $t \in [0,1]$
    \begin{equs}   \label{eq:1/4bound_by_0_in_Lp}
          \|u^\eps\|_{\mathcal{C}^{0,1/4}_p([0,t])}^2 \lesssim 1+ \|\psi\|_{C^{1/4}}^2+ \|u^\eps\|_{\mathcal{C}^{0,0}_p([0,t])}^2. 
    \end{equs}
    Next, by \eqref{eq:J-C0} and the Lipschitz continuity of $g$ again, we have
    \begin{equs}
       \|u^\eps_t\|_{\mathcal{C}^{0}_p}^2 &  \lesssim \|\psi\|_{C^{1/4}}^2 +  \int_0^t (t-r)^{-3/4} \|g(u^\eps)\|_{\mathcal{C}^{0,0}_p([0,r])}^2\,  dr
   + \int_0^t (t-r)^{-1/2} [g(u^\eps)]_{\mathcal{C}^{0, 1/4}_p([0,r])}^2\, dr
         \\
          &  \lesssim \|\psi\|_{C^{1/4}}^2 + 1+ \int_0^t (t-r)^{-3/4} \|u^\eps\|_{\mathcal{C}^{0,0}_p([0,r])}^2 \,  dr
       + \int_0^t (t-r)^{-1/2} \|u^\eps\|_{\mathcal{C}^{0, 1/4}_p([0,r])}^2\, dr.
    \end{equs}
    By using \eqref{eq:1/4bound_by_0_in_Lp} to bound the last integral, we get 
    \begin{equs}
         \|u^\eps_t\|_{\mathcal{C}^{0}_p}^2 \lesssim 1+ \|\psi\|_{C^{1/4}}^2+\int_0^t (t-r)^{-3/4} \|u^\eps\|_{\mathcal{C}^{0,0}_p([0,r])}^2 \,  dr.
    \end{equs}
    Applying Lemma \ref{lem:Gronwal} once again gives that 
    \begin{equs}
        \|u^\eps\|_{\mathcal{C}^{0, 0}_p([0,1])}^2\lesssim 1+ \|\psi\|_{C^{1/4}}^2.
    \end{equs}
Finally, plugging this bound back in \eqref{eq:1/4bound_by_0_in_Lp} finishes the proof. 
\end{proof}
\begin{corollary}           \label{lem:regularity_increments}
       Let $p \in [2, \infty)$, $\alpha \in [0, 1/4]$.  There exists a constant $N=N(p, \alpha,C_g)$ such that for all $\eps>0$ and all $(s,t)\in [0,1]^2_{\leq}$ the following bound holds:
       \begin{equs}
           \|u^\eps_t-P_{t-s}u^\eps_s\|_{\mathcal{C}^\alpha_p} \leq N(1+\|\psi\|_{C^{1/4}})|t-s|^{1/8- \alpha/2}.
       \end{equs}
\end{corollary}
\begin{proof}
    For $\alpha=1/4$, the bound follows directly by the triangle inequality, \eqref{eq:HK_omega_holder}, and  Theorem \ref{lem:niform_bound_1/4}. 
    For $\alpha=0$, since $u^\eps_t-P_{t-s}u^\eps_s= J_{s,t}^\eps[g(u^\eps)]$, the desired bound follows by \eqref{eq:J-C0}, the Lipschitz continuity of $g$,  and Theorem \ref{lem:niform_bound_1/4}. Finally, for $\alpha \in (0,1/4)$, the bound follows from the previous cases by the elementary interpolation inequality $\|\varphi\|_{\cC^\alpha_p}\leq\|\varphi\|_{\cC^0_p}^{1-4\alpha}\|\varphi\|_{\cC^{1/4}_p}^{4\alpha} $.
\end{proof}

\subsection{Tightness}
Theorem \ref{lem:niform_bound_1/4} already allows us to conclude tightness of the laws of the solutions.
In the limiting equation the noise will be obtained as the weak limit of the iterated integrals
\begin{equs}
  \beta^{\eps,n}_t & = \eps^{3/4} \int_0^t \int_{\mathbb{T}}e_n(x) J^\eps_{0, r}[1](x) \, \nabla \xi_\eps(dx, dr)  =  K_{0,t} [e_n J^\eps_{0, \cdot}[1]]\label{def:beta_epsilon_n}.
\end{equs}
As a first step we show tightness of their laws.
We remark that $\beta^{\eps,n}$ are martingales and
\begin{equ}\label{eq:integral-bet}
    K_{s,t}^\eps[f _\cdot J_{0,\cdot}^\eps[1]]=\sum_{n=1}^M\int_s^t(f_r,e_n)\,d\beta^{\eps,n}_r
\end{equ}
for $M \in \mathbb{N}$ and  any predictable $f:\Omega\times[0,1]\to\mathrm{span}\{e_1,\ldots,e_M\}$,   such that the above integrals exist.
We start by a very simple estimate that, for later use, is stated in more generality than needed for treating $\beta^{\eps,n}$.
\begin{lemma}    \label{lem:basic_estimate_K}
   Let $p \in [2,\infty)$, $\alpha\in[0,1]$. There exists a constant $N(\alpha, p)$ such that for all $(s,t) \in [0,1]^2_{\leq}$, all predictable random fields $f\in L_2([s,t];\cC^\alpha_p)$, and all $\eps>0$,  the following bound holds:
    \begin{equs}
     \| K^\eps_{s,t}[f]\|^2_{L_p} \leq N \eps^{2\alpha-1/2} \int_s^t \| f_r\|_{\mathcal{C}^\alpha_p}^2 \, dr. 
    \end{equs}
\end{lemma}
\begin{proof}
    By \eqref{eq:integration_white_noise}, the BDG inequality,  Minkowski's inequality  and \eqref{eq:HK_omega_deriv}, we have 
    \begin{equs}
        \| K^\eps_{s,t}[f]\|_{L_p}^2 &\lesssim  \eps^{3/2}\int_s^t \int_{\mathbb{T}} \|\nabla P_{\eps^2}f_r(x)\|_{L_p}^2\, dx dr 
        \\&\leq     \eps^{3/2}\int_s^t \|\nabla P_{\eps^2}f_r\|_{\mathcal{C}^0_p}^2 \, dr
        \\
        &\lesssim \eps^{2\alpha-1/2} \int_s^t \| f_r\|_{\mathcal{C}^\alpha_p}^2 \, dr.
    \end{equs}
\end{proof}

\begin{lemma}\label{lem:noise-compactness}
Let $p\in[2,\infty)$, $\gamma>3/2$. There exists a constant $N=N(\gamma, p)$ such that for any $\eps>0$ the following bound holds:
    \begin{equs}\label{eq:noise-bound}
      \sum_{n=1}^\infty n^{-\gamma} \sup_{0 \leq s<t \leq 1} \frac{\|   \beta^{\eps, n}_t-  \beta^{\eps, n}_s\|^2_{L_p}}{|t-s|}   \leq  N.
    \end{equs}
\end{lemma}
\begin{proof}
    By the definition of $\beta^{\eps,n}$  and Lemma \ref{lem:basic_estimate_K} applied with $\alpha=1/4$, we get 
    \begin{equs}
        \|   \beta^{\eps, n}_t-  \beta^{\eps, n}_s\|^2_{L_p}  =\|K_{s,t}[e_n J^\eps_{0, \cdot}[1]]\|^2_{L_p}
        & \lesssim \|e_n J^\eps_{0, \cdot}[1]]\|_{\mathcal{C}^{0,1/4}_p([0,1])}^2 |t-s|
        \\
        & \lesssim    \|e_n\|_{C^{1/4}}^2\| J^\eps_{0, \cdot}[1]]\|_{\mathcal{C}^{0,1/4}_p([0,1])}^2 |t-s|
        \\
        & \lesssim n^{1/2}|t-s|,
    \end{equs}
    where for the last inequality we have used \eqref{eq:growth_basis} and the fact  that $\| J^\eps_{0, \cdot}[1]]\|_{\mathcal{C}^{0,1/4}_p([0,1])}\lesssim 1$, which in turn follows from \eqref{eq:J-C1/4}. Multiplying with $n^{-\gamma}$ and summing over $n\in \mathbb{N}$ yields \eqref{eq:noise-bound} and finishes the proof.    
\end{proof}
For a normed space $X$ and $\alpha\in\R$, let $\ell_{2,\alpha}(\mathbb{N};  X)$ denote the space of all  $X$-valued sequences, $\lambda=(\lambda_n)_{n=1}^\infty$ such that 
\begin{equs}
    |\lambda|^2_{\ell_{2,\alpha}(\mathbb{N}; X)}:= \sum_{n=1}^\infty n^{-\alpha}\|\lambda_n\|_{X}^2 < \infty. 
\end{equs}

Let us  set 
\begin{equs}
\mathfrak{C}^+ :&=  C^{1/100}([0,1] \times \mathbb{T}) \times \ell_{2,7/4}(\mathbb{N};  C^{1/100}([0,1]))  \times \ell_{2,7/4}(\mathbb{N};  C^{1/100}([0,1])),
\\
\mathfrak{C} :&=  C([0,1] \times \mathbb{T}) \times \ell_{2,2}(\mathbb{N};  C([0,1]))  \times \ell_{2,2}(\mathbb{N};  C([0,1])).
\end{equs}
By Lemmata  \ref{lem:regularity_increments},  \ref{lem:noise-compactness} with sufficiently large $p$ and Kolmogorov's continuity theorem, the $L_p$ moment of the triple $(u^{\eps}, w, \beta^{\eps})$ as a $\mathfrak{C}^+$-valued random variable is bounded uniformly in $\eps\in(0,1]$.
Moreover, $\mathfrak{C}^+$ is compactly embedded into $\mathfrak{C}$ and the latter is a Polish space. Therefore, Prokhorov's and Skorohod's theorems readily imply the following.
\begin{lemma}\label{lem:noise-conv2}
Let Assumption \ref{as:g} hold.
  Given any sequence $(\eps_k)_{k \in \mathbb{N}}\subset(0,1)$ such that $\eps_k \to 0$, there exists a subsequence   $(\eps_k)_{k \in \mathbb{N}_1}$, $\mathbb{N}_1 \subset \mathbb{N}$,  a complete probability space $(\widehat{\Omega},\widehat{\mathcal{F}},\widehat{\P}) $  and $\mathfrak{C}$-valued random variables $ (\widehat{u}^{\eps_k}, \widehat{w}^{\eps_k}, \widehat{\beta}^{\eps_k})$, $k\in\N_1$, $(\widehat{u}, \widehat{w}, \widehat{\beta})$ defined on $\widehat{\Omega}$ such that 
  $(\widehat{u}^{\eps_k}, \widehat{w}^{\eps_k}, \widehat{\beta}^{\eps_k})\overset{\mathrm{law}}{=}(u^{\eps_k}, w, \beta^{\eps_k})$ for all $k\in\N_1$
  and $\widehat{\mathbb{P}}$-almost surely
  $(\widehat{u}^{\eps_k}, \widehat{w}^{\eps_k}, \widehat{\beta}^{\eps_k})
     \to  (\widehat{u}, \widehat{w}, \widehat{\beta})$ as $\N_1\ni k\to\infty$.
\end{lemma}

Let us introduce the following $\sigma$-algebras
\begin{equs}
    \widetilde{\mathcal{G}}^{\eps_k}_t&= \sigma \big( \widehat{u}^{\eps_k}_s(x), \widehat{w}^{\eps_k, n}_s, \widehat{\beta}^{\eps_k,n}_s ;  \, x \in \mathbb{T}, n \in \mathbb{N}, s \leq t \big), &\qquad &t \in [0,1],
    \\
       \widetilde{\mathcal{G}}_t&= \sigma \big( \widehat{u}_s(x), \widehat{w}^n_s, \widehat{\beta}^n_s ;  \, x \in \mathbb{T}, n \in \mathbb{N}, s \leq t \big), &\qquad &t \in [0,1]. 
\end{equs}
Let $\widehat{\mathcal{G}}^{\eps_k}_t$ and $\widehat{\mathcal{G}}_t$ be the $\widehat{\mathbb{P}}$-completions of $\widetilde{\mathcal{G}}^{\eps_k}_t$ and $\widetilde{\mathcal{G}}_t$, respectively and denote  $\widehat{\mathbb{G}}^{\eps_k}:= (\widehat{\mathcal{G}}^{\eps_k}_t)_{t \in [0,1]}$ and $\widehat{\mathbb{G}}:= (\widehat{\mathcal{G}}_t)_{t \in [0,1]}$. One can easily see that $ (\widehat{w}^{\eps_k, n})_{n=1}^\infty$ and  $ (\widehat{w}^{n})_{n=1}^\infty$ are sequences of independent standard $\widehat{\mathbb{G}}^{\eps_k}$ and $\widehat{\mathbb{G}}$, respectively, -Brownian motions. 
Further, let us introduce the notation\footnote{The reason to not use the notation $\widehat{\xi}_{\eps_k}$ is that in accordance to our previous notation we should have  $\widehat{\xi}_{\eps_k}=P_{{\eps_k^2}}\widehat{\xi}$ where $\widehat{\xi}= \sum_n e_n \D_t \widehat{w}^n$. Clearly,  $\widehat{\xi}_{[\eps_k]}\neq \widehat{\xi}_{\eps_k}$ but of course one has $\widehat{\xi}_{[\eps_k]}\stackrel{\mathrm{law}}{=}\widehat{\xi}_{\eps_k}$.}
\begin{equs}
\widehat{\xi}_{[\eps_k]}       = \sum_{n=1}^\infty (P_{\eps^2_k}e_n) \D_t \widehat{w}^{\eps_k,n}, 
     \end{equs}
     and notice  that $\widehat{\xi}_{[\eps_k]}\stackrel{\mathrm{law}}{=}\xi_{\eps_k}$.
     Moreover, we denote by $\widehat{K}^{\eps_k}_{s,t}[f]$, $\widehat{J}^{\eps_k}_{s,t}[f](x)$
      stochastic integrals as in \eqref{eq:K-J-notation} but with $\xi_\eps$ replaced by $\widehat{\xi}_{[\eps_k]}$.  We keep the notation $L_p$, $\mathcal{C}^\alpha_p$, and $\mathcal{C}^{0,\alpha}_p$ for random variables and random fields defined on $\widehat{\Omega}$.

     \begin{lemma}\label{lem:skorohod-corrolary}
         With the above notation one has the following:
         \begin{enumerate}[(i)]
         \item For any $p \in [1, \infty)$, there exists a constant $N=N(p, \|\psi\|_{C^{1/4}}, C_g)$ such that 
     \begin{equs} \label{eq:limit_C_1/4}
         \| \widehat{u}\|_{\mathcal{C}^{0,1/4}_p([0,1])} & \leq N,
     \end{equs}
     and for all $(s,t) \in [0,1]^2_{\leq}$, it holds that 
     \begin{equs}           \label{eq:limit_time_reg}
        \| \widehat{u}_t-P_{t-s} \widehat{u}_s\|_{\mathcal{C}^0_p} \leq N |t-s|^{1/8};
     \end{equs}
         \item For any $p\in[1,\infty)$, $(\widehat{u}^{\eps_k}, \widehat{w}^{\eps_k}, \widehat{\beta}^{\eps_k})
     \to  (\widehat{u}, \widehat{w}, \widehat{\beta})$ in $L_p(\widehat{\Omega};\mathfrak{C} )$ as $ \mathbb{N}_1 \ni k \to \infty$;
             \item $ \widehat{\beta}^{\eps_k,n}_t =\widehat K_{0,t}^{\eps_k}[e_n\widehat{J}^{\eps_k}_{0,\cdot}[1]]$;
             \item $\widehat{u}^{\eps_k}$ is the unique solution of 
             \begin{equs}\label{eq:hat-SPDE}
    \D_t \widehat{u}^{\eps_k}= \Delta \widehat{u}^{\eps_k}+ \eps_k^{3/4} g(\widehat{u}^{\eps_k}) \nabla  \widehat{\xi}_{[\eps_k]},\qquad \widehat{u}|_{t=0}=\psi.
\end{equs}
         \end{enumerate}
     \end{lemma}
     \begin{proof}
From Theorem \ref{lem:niform_bound_1/4} and Corollary \ref{lem:regularity_increments} and the equality $u^{\eps_k}\stackrel{\mathrm{law}}{=}\widehat{u}^{\eps_k}$ it follows that \eqref{eq:limit_C_1/4}-\eqref{eq:limit_time_reg} are satisfied by $\widehat{u}^{\eps_k}$. They are then also satisfied by $\widehat{u}$ by Fatou's lemma. This shows (i). The almost sure convergence and the uniform integrability implies (ii). As for (iii), one can express $\widehat K_{0,t}^{\eps_k}[e_n\widehat{J}^{\eps_k}_{0,\cdot}[1]]=\mathcal{K}^{\eps_k}(\widehat{w}^{\eps_k})$ with some measurable function $\mathcal{K}^{\eps_k}$ on $\ell_{2,2}(\N;C([0,1]))$, and moreover the equality $\beta^{\eps_k}=\mathcal{K}^{\eps_k}(w)$ holds by definition. Therefore (iii) follows by $(\widehat{\beta}^{\eps_k},\widehat{w}^{\eps_k})\stackrel{\mathrm{law}}{=}(\beta^{\eps_k},w)$.
Finally, (iv) follows similarly, noting that thanks to strong existence, the solution of \eqref{eq:hat-SPDE} can also be written in the form $\mathcal{U}^{\eps_k} (\widehat{w}^{\eps_k})$
 with some measurable function $\mathcal{U}^{\eps_k}$ on $\ell_{2,2}(\N;C([0,1]))$ and $u^{\eps_k}=\mathcal{U}^{\eps_k}(w)$ holds also by definition.
     \end{proof}
For $\phi \in C^\infty(\mathbb{T})$,  let us set 
\begin{equs}
    \widehat{M}^{\eps_k}_t(\phi)&:= (\widehat{u}^{\eps_k}_t, \phi)-(\psi, \phi)-\int_0^t (\widehat{u}^{\eps_k}_r, \Delta \phi) \, dr,
    \\
\widehat{M}_t(\phi)&:= (\widehat{u}_t, \phi)-(\psi, \phi)-\int_0^t (\widehat{u}_r, \Delta \phi) \, dr.                      \label{eq:def_hat_M}
\end{equs}
\begin{lemma}       \label{lem:M_in_Lp}
    The following hold.
    \begin{enumerate}[(i)]
        \item For each $k \in \mathbb{N}_1$ and  $\phi \in C^\infty(\mathbb{T})$,  $(\widehat{M}^{\eps_k}_t(\phi))_{t \in [0,1]}$ is a  $\widehat{\mathbb{G}}^{\eps_k}$-martingale and $\| \widehat{M}^{\eps_k}_1(\phi)\|_{L_p} < \infty$ for any $p \in [1, \infty)$. 
        \item For each $\phi \in C^\infty(\mathbb{T})$,  $(\widehat{M}_t(\phi))_{t \in [0,1]}$ is a  $\widehat{\mathbb{G}}$-martingale and $\|\widehat{M}_1(\phi)\|_{L_p} < \infty$   for any $p \in [1, \infty)$. 
    \end{enumerate}
\end{lemma}
\begin{proof}
    First note that for any $p \in [1, \infty)$, by the definition of $\widehat{M}^{\eps_k}_t(\phi)$, Minkowski's inequality, and the fact that $\widehat{u}^{\eps_k}_0=\psi$,  we have for any $t \in [0,1]$ 
    \begin{equs} \label{eq:M_eps_k_Lp}
        \|\widehat{M}^{\eps_k}_t(\phi)\|_{L_p} \leq  3\|\widehat{u}^{\eps_k}\|_{\mathcal{C}^{0,0}_p([0,1])} \sup_{x\in \mathbb{T}}(| \phi(x)| +| \Delta \phi(x)|)\leq  N_0 < \infty,
    \end{equs}
    where we have used Theorem \ref{lem:niform_bound_1/4} and $N_0$ depends only on $p, g, \psi $ and $\phi$. 
   Moreover, since $\widehat{u}^{\eps_k}$ satisfies \eqref{eq:hat-SPDE}, we clearly have that $\widehat{M}^{\eps_k}(\phi)$ is a  $\widehat{\mathbb{G}}^{\eps_k}$-martingale. Consequently, we have $(i)$. For $(ii)$, since $\widehat{u}^{\eps_k}\to \widehat{u}$ uniformly on $[0,1]\times \mathbb{T}$ almost surely, we have that
   \begin{equs}   \label{eq:lim_Meps_M}
       \widehat{\mathbb{P}} \Big( \lim_{ \mathbb{N}_1 \ni k \to \infty}\widehat{M}^{\eps_k}_t(\phi)= \widehat{M}_t(\phi) \Big)=1
   \end{equs}
   for each $t \in [0,1]$. Hence, letting $\mathbb{N}_1 \ni k \to \infty$ in \eqref{eq:M_eps_k_Lp} and using Fatou's lemma shows that $ \|\widehat{M}_t(\phi)\|_{L_p} < \infty$ for any $t \in [0,1]$. Further,  $\widehat{M}(\phi)$ is adapted to $\widehat{\mathbb{G}}$ by definition of the latter. Finally, we show the martingale property. For this, fix $(s,t)\in[0,1]_\leq^2$ and
take $\ell\in\N$, $r_1,\ldots,r_\ell\in[0,s]$, $x_1,\ldots,x_\ell\in\T$, $n_1,\ldots,n_\ell\in\N$, and a bounded continuous function $g:\R^{3\ell}\to[0,\infty)$. 
Define 
\begin{equs}
    \widehat{G}&=g(\widehat u_{r_1}(x_1),\ldots,\widehat u_{r_\ell}(x_\ell),\widehat{w}^{n_1}_{r_1},\ldots,\widehat{w}^{n_1}_{r_1},\widehat{\beta}^{n_1}_{r_1},\ldots,\widehat{\beta}^{n_1}_{r_1})       \label{eq:def_G_test}
    \\
    \widehat{G}^{\eps_k}&= g(\widehat u_{r_1}^{\eps_k}(x_1),\ldots,\widehat u_{r_\ell}^{\eps_k}(x_\ell),\widehat{w}^{\eps_k, n_1}_{r_1},\ldots,\widehat{w}^{\eps_k, n_1}_{r_1},\widehat{\beta}^{\eps_k, n_1}_{r_1},\ldots,\widehat{\beta}^{\eps_k, n_1}_{r_1})         \label{eq:def_G_eps_k_test}
\end{equs}
Since $\widehat{M}^{\eps_k}(\phi)$ is $\widehat{\mathbb{G}}^{\eps_k}$-martingale, we have 
\begin{equs}
    \widehat{\mathbb{E}}\Big(  \widehat{G}^{\eps_k} \big(\widehat{M}^{\eps_k}_t(\phi)-\widehat{M}^{\eps_k}_s(\phi)\big)\Big)=0.
\end{equs}
By letting  $\mathbb{N}_1 \ni k \to \infty$ and by uniform integrability,  we arrive at 
\begin{equs}
    \widehat{\mathbb{E}}\Big(  \widehat{G} \big(\widehat{M}_t(\phi)-\widehat{M}_s(\phi)\big)\Big)=0. 
\end{equs}
Further,  by standard measure theoretic arguments the above can be extended to any bounded $\widehat{\mathcal{G}}_s$-measurable random variable $\widehat{G}$. 
This implies that $\widehat{\mathbb{E}} \big( \widehat{M}_t(\phi) \big| \widehat{\mathcal{G}}_s \big) =\widehat{M}_s(\phi)$, which brings the proof to an end.  
\end{proof}

\section{Identification of the limit}\label{sec:limiting}
Throughout this section we assume the setting and notation of Lemma \ref{lem:noise-conv2}.
\subsection{The limiting noise}
\begin{lemma}
    The family $(8\sqrt[\leftroot{2}\uproot{2}4]{\pi}\widehat{\beta}^n)_{n=1}^\infty$ is a sequence of independent $\widehat{\mathbb{G}}$-Brownian motions.
\end{lemma}
\begin{proof}
We aim to show that $\widehat{\beta}^n$ are continuous $\widehat{\mathbb{G}}$-martingales and that for all $(s,t)\in[0,1]_\leq^2$
\begin{equ}
\widehat{\E}\big(\widehat{\beta}^n_t\widehat{\beta}^m_t-\widehat{\beta}^n_s\widehat{\beta}^m_s|\widehat{\mathcal{G}}_s\big)
    =\mathbf{1}_{m=n}
    c_0(t-s)\label{eq:var-first}
\end{equ}
with $c_0=(64\sqrt{\pi})^{-1}$.
By L\'evy's characterisation this is sufficient.
Fix $(s,t)\in[0,1]_\leq^2$ and
take  $\widehat{G}$ and $\widehat{G}^{\eps_k}$ of the form  \eqref{eq:def_G_test}-\eqref{eq:def_G_eps_k_test},  and define  $G^\eps$ analogously.
It suffices to show that for any choice of these objects
\begin{equs}
\widehat{\E}\big(\widehat{G}(\widehat{\beta}^n_t-\widehat{\beta}^n_s)\big)&=0,\label{eq:var-easy}
\\
    \widehat{\E}\big(\widehat{G}(\widehat{\beta}^n_t-\widehat{\beta}^n_s)(\widehat{\beta}^m_t-\widehat{\beta}^m_s)\big)&=(\widehat{\E} \widehat{G})\mathbf{1}_{m=n}c_0(t-s).\label{eq:var-hard}
\end{equs}
Indeed, \eqref{eq:var-easy} is precisely the martingale property. Furthermore,
\begin{equ}\label{eq:var-decomp}
   \widehat{\E}\big(\widehat{G}(\widehat{\beta}^n_t-\widehat{\beta}^n_s)(\widehat{\beta}^m_t-\widehat{\beta}^m_s)\big)
   =
 \widehat{\E}\big(\widehat{G}(\widehat{\beta}^n_t\widehat{\beta}^m_t-\widehat{\beta}^n_s\widehat{\beta}^m_s)\big)
-\widehat{\E}\big(\widehat{G}(\widehat{\beta}^n_t-\widehat{\beta}^n_s)\widehat{\beta}^m_s\big)-
\widehat{\E}\big(\widehat{G}(\widehat{\beta}^m_t-\widehat{\beta}^m_s)\widehat{\beta}^n_s\big).
\end{equ}
The latter two terms are easily seen to be $0$, by applying \eqref{eq:var-easy} with $\widehat{G}':=\widehat{G}\beta^m_s$ and $\widehat{G}'':=\widehat{G}\beta^n_s$ in place of $\widehat{G}$, and a standard limiting argument (since $\widehat{G}'$, $\widehat{G}''$ are not bounded). Combining \eqref{eq:var-decomp} with \eqref{eq:var-hard}, we get \eqref{eq:var-first}.

Note that by Lemma \ref{lem:skorohod-corrolary}
\begin{equs}
    \widehat{\E}\big(\widehat{G}(\widehat{\beta}^n_t-\widehat{\beta}^n_s)\big)&=\lim_{\N_1\ni k\to \infty}\widehat{\E}\big(\widehat{G}^{\eps_k}(\widehat{\beta}^{\eps_k,n}_t-\widehat{\beta}^{\eps_k,n}_s)\big)
    \\
    &=\lim_{\N_1\ni k\to \infty}\E \big({G}^{\eps_k}({\beta}^{\eps_k,n}_t-{\beta}^{\eps_k,n}_s)\big)=0,
\end{equs}
since $\beta^{\eps_k,n}$ is by definition a $\mathbb{F}$-martingale, see \eqref{def:beta_epsilon_n}, and $G^{\eps_k}$ is $\mathcal{F}_s$-measurable.

Moving on to \eqref{eq:var-hard}, again by Lemma \ref{lem:skorohod-corrolary}
\begin{equs}
    \widehat{\E}\big(\widehat{G}(\widehat{\beta}^n_t-\widehat{\beta}^n_s)(\widehat{\beta}^m_t-\widehat{\beta}^m_s)\big)&=\lim_{\N_1\ni k\to \infty}\widehat{\E}\big(\widehat{G}^{\eps_k}(\widehat{\beta}^{\eps_k,n}_t-\widehat{\beta}^{\eps_k,n}_s)(\widehat{\beta}^{\eps_k,m}_t-\widehat{\beta}^{\eps_k,m}_s)\big)
    \\
    &=\lim_{\N_1\ni k\to \infty}\E\big(G^{\eps_k}(\beta^{\eps_k,n}_t-\beta^{\eps_k,n}_s)(\beta^{\eps_k,m}_t-\beta^{\eps_k,m}_s)\big).
\end{equs}
Recalling that  $\beta^{\eps,n}_t=K^\eps_{0,t}[e_n J_{0,\cdot}^\eps[1]]$, by polarisation it suffices to show for every $f\in C^1(\T)$
\begin{equ}\label{eq:var-goal}
    \lim_{\N_1\ni k\to \infty}\E\big(G^{\eps_k} (K^{\eps_k}_{s,t}[fJ^{\eps_k}_{0,\cdot}[1]])^2\big)=c_0(t-s)\|f\|_{L^2(\T)}^2\widehat{\E}\, \widehat{G}. 
\end{equ}

Recalling \eqref{eq:integration_white_noise},  by It\^o's isometry we get
\begin{equs}
    \E\big(\sqrt{G^\eps}K_{s,t}^\eps[fJ_{0,\cdot}^\eps[1]]\big)^2&=\eps^{3/2}\E\Big(\int_s^t\int_\T\int_\T \sqrt{G^\eps}\nabla P_{\eps^2}(z-x)f(x)J_{0,r}^\eps[1](x)\,dx\xi(dz,dr)\Big)^2
    \\
    &=\int_s^t\int_\T\eps^{3/2}\E\Big(\sqrt{G^\eps}\int_\T\nabla P_{\eps^2}(z-x)f(x)J_{0,r}^\eps[1](x)\,dx\Big)^2\,dzdr.\label{eq:variance-1}
\end{equs}
Define the quantities $A_1,A_2,A_2$ as
\begin{equs}
    A_1:&=\eps^{3/4}\sqrt{G^\eps}\int_\T\nabla P_{\eps^2}(z-x)f(z)J_{s,r}^\eps[1](x)\,dx,
\\
A_2:&=\eps^{3/4}\sqrt{G^\eps}\int_\T\nabla P_{\eps^2}(z-x)\big(f(x)-f(z)\big)J_{s,r}^\eps[1](x)\,dx,
\\
A_3:&=\eps^{3/4}\sqrt{G^\eps}\int_\T\nabla P_{\eps^2}(z-x)f(x)P_{s-r}(J_{0,s}^\eps[1])(x)\,dx.
\end{equs}
Using the definition of $J^{\eps}_{0,r}[1]$, the independence of $\mathcal{F}_s$ and $\xi|_{[s,t]\times\T}$, the semigroup property, and It\^o's isometry we get
\begin{equs}
    \E A_1^2:&=\eps^{3/2}\E\Big(\sqrt{G^\eps}\int_\T\nabla P_{\eps^2}(z-x)f(z)J_{s,r}^\eps[1](x)\,dx\Big)^2
    \\
    &=(f(z))^2\eps^{3}\E\Big(\sqrt{G^\eps}\int_\T \nabla P_{\eps^2}(z-x)\int_s^r\int_\T\nabla P_{\eps^2+r-\tau}(x-y) \xi(dy,d\tau)\,dx\Big)^2  
    \\
    &=(f(z))^2\eps^{3}\E(G^\eps)\E\Big(\int_\T \nabla P_{\eps^2}(z-x)\int_s^r\int_\T\nabla P_{\eps^2+r-\tau}(x-y) \xi(dy,d\tau)\,dx\Big)^2
    \\
    &=(f(z))^2\eps^{3}\E(G^\eps)\E\Big(\int_s^r\int_\T\nabla^2 P_{2\eps^2+r-\tau}(z-y) \xi(dy,d\tau)\Big)^2
    \\
    &=(f(z))^2\eps^{3}\E(G^\eps)\int_s^r\|\nabla^2 P_{2\eps^2+r-\tau}\|_{L_2(\T)}^2\,d\tau.
\end{equs}
Note that the standard heat kernel bound \eqref{eq:HK_0} already provides an  upper bound of order $1$ for this quantity. In order to get the exact constant first note that
\begin{equ}
    \|\nabla^2 P_t\|_{L_2(\T)}^2=\sum_{k\in\Z}(2\pi k)^4e^{-8\pi^2 k^2t}=(2\pi)^4t^{-2}\sum_{x\in \sqrt{t}\Z}x^4e^{-8\pi^2 x^2}=
    (2\pi)^4t^{-5/2}\Big(\int_\R x^4e^{-8\pi^2 x^2}\,dx+R\Big),
\end{equ}
where the error term $R$ can be estimated by
\begin{equ}
    |R|\leq \sqrt{t}\int_\R \big|\partial_x\big(x^4e^{-8\pi^2 x^2}\big)\big|\,dx\lesssim \sqrt{t}.
\end{equ}
Calculating the Gaussian integral explicitly
\begin{equ}
    \int_\R x^4e^{-8\pi^2 x^2}\,dx=\frac{3}{512\sqrt{2}\pi^{9/2}},
\end{equ}
we get
\begin{equ}
     \|\nabla^2 P_t\|_{L_2(\T)}^2=\frac{3}{32\sqrt{2\pi}}t^{-5/2}+O(t^{-2}),
\end{equ}
and so
\begin{equs}
    \E A_1^2&=\E(G^\eps)(f(z))^2\eps^{3}\Big(\int_0^r\frac{3}{32\sqrt{2\pi}}(2\eps^2+r-\tau)^{-5/2}\,d\tau+O(\eps^{-2})\Big)
    \\
    &=\E(G^\eps)(f(z))^2\eps^{3}\Big(\frac{1}{16\sqrt{2\pi}}\big( (2\eps^2)^{-3/2}-(2\eps^2+r)^{-3/2}\big)+O(\eps^{-2})\Big)\\
    &=\E(G^\eps)(f(z))^2\Big(c_0+O(\eps^{1/2}r^{-1/4})+O(\eps)\Big),
\end{equs}
recalling the notation $c_0=(64\sqrt{\pi})^{-1}$.

Next, for $A_2$
one can simply apply Minkowski's inequality, the upper bound on $G^\eps$, and the bound $\|J_{s,\cdot}^\eps[1]\|_{\cC^{0,0}_2([0,1])}\lesssim 1$ to get
\begin{equ}
    \|A_2\|_{L^2}\lesssim \eps^{3/4}\int_\T|\nabla P_{\eps^2}(z-x)||f(x)-f(z)|\,dx\stackrel{\eqref{eq:HK-power}}{\lesssim}\eps^{3/4}\|f\|_{C^1(\T)}.
\end{equ}

Finally, for $A_3$ we integrate by parts and proceed with Minkowski's inequality
\begin{equs}
    \|A_3\|_{L^2}&\leq\eps^{3/4}\Big\|
    (\sqrt{G^\eps}\int_\T P_{\eps^2}(z-x)\nabla f(x)P_{s-r}(J_{0,s}^\eps[1])(x)\,dx
    \Big\|_{L_2}
    \\
    &\quad+\eps^{3/4}\Big\|
    (\sqrt{G^\eps}\int_\T P_{\eps^2}(z-x) f(x)\nabla P_{s-r}(J_{0,s}^\eps[1])(x)\,dx
    \Big\|_{L_2}
    \\
    &\lesssim \eps^{3/4}\|f\|_{C^1(\T)}\big(\|P_{s-r}J_{0,s}^\eps[1]\|_{\cC^{0}_2}+
    \|\nabla P_{s-r}J_{0,s}^\eps[1]\|_{\cC^{0}_2}\big)   
    \\
    &\stackrel{\eqref{eq:HK_omega_deriv}}{\lesssim} \eps^{3/4}\|f\|_{C^1(\T)}(r-s)^{-3/8}\|J_{0,s}^\eps[1]\|_{\cC^{1/4}_2}
    \\
    &\lesssim \eps^{3/4}\|f\|_{C^1(\T)}(r-s)^{-3/8},
\end{equs}
where in the last inequality we used \eqref{eq:J-C1/4} with $f\equiv 1$.

Substituting back into \eqref{eq:variance-1}, we arrive to
\begin{equs}
    \E\big(K_{0,t}[fJ_{0,\cdot}^\eps[1]]\big)^2&=\int_s^t\int_\T \E\big((A_1+A_2+A_3)^2\big)\,dzdr\\
    &=c_0(t-s)\|f\|_{L_2(\T)}^2\E(G^\eps)+O(\eps^{1/2}).
\end{equs}
Further, using this along the sequence $(\eps_k)_{k\in\N_1}$, and recalling that $G^{\eps_k}\overset{\mathrm{law}}{=}\widehat{G}^{\eps_k}$,  we have 
\begin{equs}
    \E\big(K_{0,t}[fJ_{0,\cdot}^{\eps_k}[1]]\big)^2&=c_0(t-s)\|f\|_{L_2(\T)}^2\widehat{\E}(\widehat{G}^{\eps_k})+O(\eps_k^{1/2}).
\end{equs}
By letting $\mathbb{N}_1 \ni k\to \infty$, we get \eqref{eq:var-goal} and the proof is finished.
\end{proof}

\subsection{Sketch of the strategy}

Recalling \eqref{eq:def_hat_M}, we would like to show that 
\begin{equs}
    \widehat{M}_t(\phi) = \sum_{n =1}^\infty \int_0^t (g'g(\widehat{u}_r) e_n, \phi) \, d\widehat{\beta}^n_r.\label{eq:martingale-goal}
\end{equs}
Since the tools at our disposal are the convergences $\widehat{u}^{\eps_k}\to\widehat{u}$ and $\widehat{\beta}^{n,\eps_k}\to\widehat{\beta}^n$, we would like to recognise the stochastic integrals
\begin{equ}
    \sum_{n =1}^\infty \int_0^t (g'g(\widehat{u}^{\eps_k}_r) e_n, \phi) \, d\widehat{\beta}^{n,\eps_k}_r
\end{equ}
in the martingales $\widehat{M}^{\eps_k}(\phi)$. 

It turns out to be more convenient to work with temporally ``frozen" integrands.
This will be done several steps below, but let us first motivate why this is sufficient to achieve \eqref{eq:martingale-goal}. Since $t \mapsto \widehat{u}_t$ has positive temporal regularity in the sense of  \eqref{eq:limit_time_reg}, it is a consequence of the the stochastic sewing lemma (see \cite[Example~2.10]{SSL}) that the right-hand side of \eqref{eq:martingale-goal} can be characterised as the unique martingale $\mathcal{M}$ that satisfies with some $\eps>0$ and $N<\infty$ for all $(s,t)\in[0,1]_{\leq}^2$
\begin{equs}        \label{eq:est1/2}
   \|\mathcal{M}_t- \mathcal{M}_s - \sum_{n=1}^\infty \int_s^t \big(g'g(P_{r-s}\widehat{u}_s) e_n, \phi\big) \, d\widehat{\beta}^n_r\|_{L_2} \leq N |t-s|^{1/2+\eps}.
\end{equs}
Hence it suffices to prove \eqref{eq:est1/2} for $\mathcal{M}=\widehat{M}(\phi)$.  It turns out that this can be achieved with $\eps=1/8$. This argument  will be made rigorous  in the proof of our main theorem. 

Notice that  \eqref{eq:est1/2} characterizes $\mathcal{M}$ by germs where the integrand is frozen in time. To show \eqref{eq:est1/2}, we will use a limiting procedure, hence, we need to freeze the corresponding quantities at the approximate level. 
First we recall that
\begin{equs}         \label{eq:identity_before_limit}
    \widehat{M}^{\eps_k}_t(\phi)-  \widehat{M}^{\eps_k}_s(\phi)= \widehat{K}^{\eps_k}_{s,t}[\phi \,  g(\widehat{u}^{\eps_k})],
\end{equs}
so as a first freezing we write
\begin{equs}
\widehat{K}^{\eps_k}_{s,t}[\phi \, g(\widehat{u}^{\eps_k})] & =    \widehat{K}^{\eps_k}_{s,t}[\phi \, g(P_{\cdot-s}\widehat{u}^{\eps_k}_s)]
  +  \widehat{K}^{\eps_k}_{s,t}[\phi \big(g(\widehat{u}^{\eps_k})- g(P_{\cdot-s}\widehat{u}^{\eps_k}_s)\big)].        \label{eq:first_decomp_K}
\end{equs}
Next, let us set
\begin{equs}
    G_{s,r}[\widehat{u}^{\eps_k}](x):=  \int_0^1 g'\big( \theta \widehat{u}^{\eps_k}_r(x)+(1-\theta) P_{r-s}\widehat{u}^{\eps_k}_s(x) \big) \, d\theta. \label{eq:definition_of_G}
\end{equs}
By the fundamental theorem of calculus we have 
\begin{equs}
    g(\widehat{u}^{\eps_k}_r(x))- g(P_{r-s}\widehat{u}^{\eps_k}_s(x)) &=   G_{s,r}[\widehat{u}^{\eps_k}](x)\big(\widehat{u}^{\eps_k}_r(x)-P_{r-s}\widehat{u}^{\eps_k}_s(x)\big)
    \\
    &=  G_{s,r}[\widehat{u}^{\eps_k}](x)\,  \widehat{J}^{\eps_k}_{s,r}[g(\widehat{u}^{\eps_k})](x)
    \\
    &=  g'(P_{r-s}\widehat{u}^{\eps_k}_s(x))\,  \widehat{J}^{\eps_k}_{s,r}[g(\widehat{u}^{\eps_k})](x)
    \\ 
    & \qquad+\big(G_{s,r}[\widehat{u}^{\eps_k}](x)-g'(P_{r-s}\widehat{u}^{\eps_k}_s(x))\big)\,  \widehat{J}^{\eps_k}_{s,r}[g(\widehat{u}^{\eps_k})](x)  .   \label{eq:decomp_g}
\end{equs}
We can further expand the second factor of the first term above as follows 
\begin{equs}
     \widehat{J}^{\eps_k}_{s,r}[g(\widehat{u}^{\eps_k})](x) 
      &=   g(P_{r-s}\widehat{u}^{\eps_k}_s(x))\widehat{J}^{\eps_k}_{0,r}[1](x)
      \\
       & \qquad - g(P_{r-s}\widehat{u}^{\eps_k}_s(x))\big(P_{r-s}\widehat{J}^{\eps_k}_{0,s}[1]\big)(x)
     \\
     & \qquad + \widehat{J}^{\eps_k}_{s,r}[g(P_{\diamond-s}\widehat{u}^{\eps_k}_s)](x) -g(P_{r-s}\widehat{u}^{\eps_k}_s)\widehat{J}^{\eps_k}_{s,r}[1](x)
     \\  & \qquad + \widehat{J}^{\eps_k}_{s,r}[g(\widehat{u}^{\eps_k})-g(P_{\diamond-s}\widehat{u}^{\eps_k}_s)](x).
\end{equs}
Replacing this in \eqref{eq:decomp_g} we get 
\begin{equs}
     g(\widehat{u}^{\eps_k}_r(x))- g(P_{r-s}\widehat{u}^{\eps_k}_s(x)) &= 
      g'g(P_{r-s}\widehat{u}^{\eps_k}_s(x))\, \widehat{J}^{\eps_k}_{0,r}[1](x)
    \\ 
    & \qquad -   g'g(P_{r-s}\widehat{u}^{\eps_k}_s(x)) \big(P_{r-s}\widehat{J}^{\eps_k}_{0,s}[1]\big)(x)
    \\
    &\qquad + g'(P_{r-s}\widehat{u}^{\eps_k}_s(x))\big( \widehat{J}^{\eps_k}_{s,r}[g(P_{\diamond-s}\widehat{u}^{\eps_k}_s)](x) -g(P_{r-s}\widehat{u}^{\eps_k}_s)\widehat{J}^{\eps_k}_{s,r}[1](x)\big)
    \\
     &\qquad + g'(P_{r-s}\widehat{u}^{\eps_k}_s(x))\widehat{J}^{\eps_k}_{s,r}[g(\widehat{u}^{\eps_k})-g(P_{\diamond-s}\widehat{u}^{\eps_k}_s)](x)
     \\
    & \qquad+\big(G_{s,r}[\widehat{u}^{\eps_k}](x)-g'(P_{r-s}\widehat{u}^{\eps_k}_s(x))\big)\,  \widehat{J}^{\eps_k}_{s,r}[g(\widehat{u}^{\eps_k})](x).    
\end{equs}
Hence, replacing this in \eqref{eq:first_decomp_K}, we get
\begin{equs}
    \widehat{K}^{\eps_k}_{s,t}[\phi \, g(\widehat{u}^{\eps_k})] & =    \widehat{K}^{\eps_k}_{s,t}\big[\phi  g'g(P_{\cdot-s}\widehat{u}^{\eps_k}_s)\, \widehat{J}^{\eps_k}_{0,\cdot}[1]\big]
    \\
    &\qquad+\widehat{K}^{\eps_k}_{s,t}[\phi \, g(P_{\cdot-s}\widehat{u}^{\eps_k}_s)]
    \\
    &\qquad - \widehat{K}^{\eps_k}_{s,t}\big[ \phi  g'g(P_{\cdot-s}\widehat{u}^{\eps_k}_s)\, \big(P_{\cdot-s}\widehat{J}^{\eps_k}_{0,s}[1]\big) \big]
    \\
     &\qquad +  \widehat{K}^{\eps_k}_{s,t}\big[ \phi  g'(P_{\cdot-s}\widehat{u}^{\eps_k}_s)\big( \widehat{J}^{\eps_k}_{s,\cdot}[g(P_{\diamond-s}\widehat{u}^{\eps_k}_s)] -g(P_{\cdot-s}\widehat{u}^{\eps_k}_s)\widehat{J}^{\eps_k}_{s,\cdot}[1]\big)   \big]
    \\
    &\qquad +  \widehat{K}^{\eps_k}_{s,t}\big[ \phi  g'(P_{\cdot-s}\widehat{u}^{\eps_k}_s)\widehat{J}^{\eps_k}_{s,\cdot}[g(\widehat{u}^{\eps_k})-g(P_{\diamond-s}\widehat{u}^{\eps_k}_s)]  \big]
    \\
    & \qquad + \widehat{K}^{\eps_k}_{s,t}\big[ \phi \big(G_{s,\cdot}[\widehat{u}^{\eps_k}]-g'(P_{\cdot-s}\widehat{u}^{\eps_k}_s)\big)\,  \widehat{J}^{\eps_k}_{s,\cdot}[g(\widehat{u}^{\eps_k})]  \big]
    \\
    &=: \sum_{m=1}^6 \widehat{\mathbb{K}}_{s,t}^{\eps_k, m}.             \label{eq:K_decomp_bold_Ks}
\end{equs}
The first of these terms will be shown to converge to the sum of stochastic integrals in \eqref{eq:est1/2}. This is the content of Lemma \ref{lem:germ_delta_estimate} below. 
Therefore, in order to obtain \eqref{eq:est1/2} with $\mathcal{M}=\widehat{M}(\phi)$, it is sufficient to prove the bounds
\begin{equ}
   \|\widehat{\mathbb{K}}_{s,t}^{\eps_k, m}\|_{L^2}  \lesssim \eps_k^{1/4}, \, \text{for $m=2,3$},\qquad  \text{and} \qquad \|\widehat{\mathbb{K}}_{s,t}^{\eps_k, m}\|_{L^2}\lesssim |t-s|^{5/8},\, \text{for $m=4,5,6$.}
\end{equ}
These bounds are proved in Lemmata \ref{lem:K_2K_3}, \ref{lem:K4}, \ref{lem:K5}, and \ref{lem:K6}  below.

\subsection{Convergence of the main term}

\begin{lemma}                       \label{lem:germ_delta_estimate}
Let $\phi:\T\to\R$ be twice continuously differentiable.
Then for all $(s,t)\in[0,1]_{\leq}^2$ one has
\begin{equs}
    \lim_{\mathbb{N}_1 \ni k\to \infty}\Big\|\widehat{\mathbb{K}}^{\eps_k,1}_{s,t}- \sum_{n=1}^\infty \int_s^t  \big(g'g(P_{r-s}\widehat{u}_s)e_n, \phi \big) \, d \widehat{\beta}^n_r \Big\|_{L_2}=0.  \label{eq:convergence_K_1}
\end{equs}
\end{lemma}
\begin{proof}

Throughout the proof,  we repeatedly use moment estimates from Section \ref{sec:apriori-1} with $\widehat{u}^{\eps_k}$, $\widehat{K}^{\eps_k}_{s,t}$, $\widehat{J}^{\eps_k}_{s,t}(x)$ in place of
${u}^{\eps_k}$, ${K}^{\eps_k}_{s,t}$, ${J}^{\eps_k}_{s,t}(x)$, which follow from $(\widehat{u}^{\eps_k},\widehat{w}^{\eps_k})\stackrel{\mathrm{law}}{=}(u^{\eps_k},w)$.
Let us introduce the notation 
\begin{equs}
   F^{\eps_k}_r(x):= \phi(x)  g'g\big(P_{r-s}\widehat{u}^{\eps_k}_s(x) \big), \qquad  F_r(x):= \phi(x)  g'g\big(P_{r-s}\widehat{u}_s(x) \big),
\end{equs}
for $r \in [s,1]$, $x \in \mathbb{T}$.
Let $\delta\in(0,t-s)$ and $M\in\N$ and denote by $\Pi_M$ the orthogonal projection from $L_2(\mathbb{T})$ on $\text{span}(e_1,\dots, e_M)$.
We write
\begin{equs}
     \Big\|\widehat{\mathbb{K}}^{\eps_k,1}_{s,t}- &\sum_{n=1}^\infty \int_s^t  \big(g'g(P_{r-s}\widehat{u}_s)e_n, \phi \big) \, d \widehat{\beta}^n_r \Big\|_{L_2}
   \\
   & \leq  \Big\| \widehat{K}^{\eps_k}_{s,s+\delta}\big[\phi  g'g(P_{\cdot-s}\widehat{u}^{\eps_k}_s)\, \widehat{J}^{\eps_k}_{0,\cdot}[1]\big]\Big\|_{L_2}
   + \Big\|\sum_{n=1}^\infty \int_s^{s+\delta}  \big(g'g(P_{r-s}\widehat{u}_s)e_n, \phi \big) \, d \widehat{\beta}^n_r \Big\|_{L_2}
    \\
    & \qquad + \Big\|\widehat{K}^{\eps_k}_{s+\delta,t}\big[ (F^{\eps_k}-\Pi_MF^{\eps_k})\, \widehat{J}^{\eps_k}_{0,\cdot}[1]\big]  \Big\|_{L_2}+ \Big\| \sum_{n=M+1}^\infty \int_{s+\delta}^t  \big(F_r, e_n \big) \, d \widehat{\beta}^n_r \Big\|_{L_2}
     \\
     & \qquad+\Big\|\widehat{K}^{\eps_k}_{s+\delta,t}\big[ (\Pi_MF^{\eps_k})\, \widehat{J}^{\eps_k}_{0,\cdot}[1]\big]- \sum_{n=1}^M \int_{s+\delta}^t  \big(F_r, e_n \big) \, d \widehat{\beta}^n_r \Big\|_{L_2}
    \\    
    & =: \sum_{i=1}^5 A_i(\delta,M,k).\label{eq:decomp_Qis}
\end{equs}

Applying Lemma \ref{lem:basic_estimate_K} with $\alpha=1/2$ yields
\begin{equs}
   A_1(\delta,M,k) & \lesssim  \|\phi  g'g(P_{\cdot-s}\widehat{u}^{\eps_k}_s)\, \widehat{J}^{\eps_k}_{0,\cdot}[1]\big] \|_{\mathcal{C}^{0,1/4}_2([s, s+\delta])} \delta^{1/2}
   \\
  & \lesssim  \| g'g(P_{\cdot-s}\widehat{u}^{\eps_k}_s)\|_{\mathcal{C}^{0,1/4}_4([s, s+\delta])}\| \widehat{J}^{\eps_k}_{0,\cdot}[1]\big] \|_{\mathcal{C}^{0,1/4}_4([s, s+\delta])} \delta^{1/2}
  \\
  & \lesssim (1+ \| \widehat{u}^{\eps_k}\|_{\mathcal{C}^{0,1/4}_8([0, 1])}^2)\|\widehat{J}^{\eps_k}_{0,\cdot}[1]\big] \|_{\mathcal{C}^{0,1/4}_4([s, s+\delta])} \delta^{1/2}
   \\
  & \lesssim \delta^{1/2}, \label{eq:Q1}
\end{equs}
where for the third inequality we have used our assumption on $g$ and \eqref{eq:HK_omega_holder}, while for  last inequality we have used Theorem \ref{lem:niform_bound_1/4} and \eqref{eq:J-C1/4}.

Next, by It\^o's isometry and Parseval's identity, we have 
\begin{equs}
    \big(A_2(\delta,M,k)\big)^2 = \int_s^{s+\delta}  \widehat{\E} \| g'g(P_{r-s}\widehat{u}_s) \phi \|^2_{L_2(\mathbb{T})} \, dr \lesssim  \int_s^{s+\delta}(1+\| \widehat{u}\|_{\mathcal{C}^{0,0}_2([0, 1])}^2) \, dr \lesssim \delta, \label{eq:Q2}
\end{equs}
using the assumption on $g$ and \eqref{eq:limit_C_1/4}.

Next, by using Lemma  \ref{lem:basic_estimate_K} and \eqref{eq:J-C0}, we get
\begin{equs}
    A_3(\delta,M,k)&\lesssim \| (F^{\eps_k}-\Pi_MF^{\eps_k})\widehat{J}^{\eps_k}_{0,\cdot}[1]\big]\|_{\mathcal{C}^{0,1/4}_2([s+\delta, t])} 
        \\
    &  \lesssim \| F^{\eps_k}-\Pi_MF^{\eps_k}\|_{\mathcal{C}^{0,1/4}_{4}([s+\delta, t])}    
    \\
    &\lesssim  \sum_{n=M+1}^\infty n^{-2} \|( \Delta F^{\eps_k}, e_n)e_n \, \|_{\mathcal{C}^{0,1/4}_{4}([s+\delta, t])}
     \\
     & \lesssim  \| \Delta F^{\eps_k}\|_{\mathcal{C}^{0,0}_{4}([s+\delta, t])} \sum_{n=M+1}^\infty n^{-7/4}.
\end{equs}
By our assumption on $g$ and $\phi$, it follows that
\begin{equs}
    |\Delta F^{\eps_k}_r(x)|\lesssim 1+\sum_{j=0}^2|\nabla ^j P_{r-s}\widehat{u}^{\eps_k}_s(x)|^4 
\end{equs}
so that upon using \eqref{eq:HK_omega_deriv}, Theorem \ref{lem:niform_bound_1/4} and the fact that $r \geq s+\delta$, we obtain 
\begin{equs}
    \| \Delta F^{\eps_k}\|_{\mathcal{C}^{0,0}_{4}([s+\delta, t])} \lesssim \delta^{-4}.
\end{equs}
Therefore,
\begin{equ}
    A_3\lesssim \delta^{-4}M^{-3/4}.
\end{equ}

Next, by It\^o's isometry
\begin{equs}
    \big(A_4(\delta,M,k)\big)^2=\int_{s+\delta}^t\widehat{\E}\|F_r-\Pi_M F_r\|_{L^2(\T)}^2\,dr\leq M^{-1/4}\int_s^t\widehat{\E}\|F_r\|_{H^{1/8}}^2\,dr,
\end{equs}
where we use the standard Sobolev norm defined as $\|f\|_{H^{s}(\T)}^2=\sum_{n\in\N}n^{2s}(f,e_n)^2$.
From the standard embedding $C^{\alpha}(\T)\subset H^{\beta}(\T)$ for all $0\leq \beta<\alpha\leq 1$ and the embedding $\mathcal{C}^{\alpha}_p\subset L_p(\widehat{\Omega},C^{\alpha-2/p}(\T))$ from
Kolmogorov's continuity theorem, we get
\begin{equ}
    A_4(\delta,M,k)\lesssim M^{-1/8}\|F\|_{\mathcal{C}^{0,1/4}_{100}}^2\lesssim M^{-1/8}.
\end{equ}

Finally, for $A_5$ first notice that
 by virtue of  of Lemma \ref{lem:noise-conv2} and the regularity of $g$,  for any $n \in \mathbb{N}$  we have 
\begin{equs}    \label{eq:F_eps_convergence}
    \sup_{r \in [0,1]}\big( |(F^{\eps_k}_r, e_n) - (F^{\eps_k}_r, e_n)|+|\widehat{\beta}^{\eps_k,n}_r-\widehat{\beta}^n_r|\big) \to 0, \qquad \text{as $\mathbb{N}_1 \ni k \to \infty$},
\end{equs}
in probability. Also, one has
\begin{equ}
    \widehat{K}^{\eps_k}_{s+\delta,t}\big[ (\Pi_MF^{\eps_k})\, \widehat{J}^{\eps_k}_{0,\cdot}[1]\big]
    = \sum_{n=1}^M \int_{s+\delta}^t  (F^{\eps_k}_r, e_n) \, d \widehat{\beta}^{\eps_k,n}_r,
\end{equ}
see \eqref{eq:integral-bet}.
 Consequently, by using \eqref{eq:F_eps_convergence} and Lemma \ref{lem:convergence_stochastic_integrals}, we get 
\begin{equs}
  \widehat{K}^{\eps_k}_{s+\delta,t}\big[ &(\Pi_MF^{\eps_k})\, \widehat{J}^{\eps_k}_{0,\cdot}[1]\big] - \sum_{n=1}^M \int_{s+\delta}^t  \big(F_r, e_n \big) \, d \widehat{\beta}^n_r 
 \to 0, \qquad \text{as $\mathbb{N}_1 \ni k \to \infty$}, \label{eq:bounded_in_Lp}
\end{equs}
in probability. Moreover, by using Lemma  \ref{lem:basic_estimate_K} and \eqref{eq:J-C0}, we get 
\begin{equs}
    \big\| \widehat{K}^{\eps_k}_{s+\delta,t}\big[ (\Pi_MF^{\eps_k})\, \widehat{J}^{\eps_k}_{0,\cdot}[1]\big] \big\|_{L_p} & \lesssim \| (\Pi_MF^{\eps_k})\widehat{J}^{\eps_k}_{0,\cdot}[1]\big]\|_{\mathcal{C}^{0,1/4}_p([s+\delta, t])} 
  \\
  & \lesssim \| \Pi_MF^{\eps_k}\|_{\mathcal{C}^{0,1/4}_{2p}([s+\delta, t])}.  
\label{eq:est_K_Proj}
\end{equs}
In addition, it is easy to see that 
\begin{equs} 
    \| \Pi_MF^{\eps_k}\|_{\mathcal{C}^{0,1/4}_{2p}([s+\delta, t])} & = \|  \sum_{n=1}^M(F^{\eps_k}, e_n)e_n\|_{\mathcal{C}^{0,1/4}_{2p}([s+\delta, t])}\lesssim  M^{5/4} \| F^{\eps_k}\|_{\mathcal{C}^{0,0}_{2p}([s+\delta, t])}
    \\
    & \lesssim  M^{5/4}(1+\| u^{\eps_k}\|_{\mathcal{C}^{0,0}_{2p}([s+\delta, t])})\lesssim  M^{5/4},
\end{equs}
where we have used Theorem \ref{lem:niform_bound_1/4} for the last inequality. Hence, the quantity in \eqref{eq:bounded_in_Lp} is bounded in $L_p$ uniformly in $k$ for any $p \geq 1 $, and by uniform integrability we obtain that for any $M$ and $\delta$, 
\begin{equs}   \label{eq:Q_3_1}
    \lim_{\mathbb{N}_1 \ni k\to \infty} A_5(\delta,M,k)=0.
\end{equs}
We conclude that for all $i=1,2,3,4,5$, one has
\begin{equ}
    \limsup_{\delta\to 0}\limsup_{M\to\infty}\limsup_{\mathbb{N}_1 \ni k\to \infty} A_i(\delta,M,k)=0,
\end{equ}
which by \eqref{eq:decomp_Qis} proves the claim.
\end{proof}

\subsection{Estimating the remainders}

\begin{lemma}        \label{lem:K_2K_3}
    For any $\phi \in C^{1/2}(\T)$,
    there exists a constant $N=N(C_g, \|\psi\|_{C^{1/4}},\|\phi\|_{C^{1/2}})$ such that for all $k\in\N_1$,
    $m \in \{2,3\}$, and $(s,t) \in [0,1]^2_{\leq}$, we have 
     \begin{equs}
        \|\widehat{\mathbb{K}}^{\eps_k,m}_{s,t}\|_{L_2}\leq N \eps_k^{1/4}.
     \end{equs}
\end{lemma}

\begin{proof}
    We consider the case $m=3$ first. Since $(\widehat{u}^{\eps_k},\widehat{w}^{\eps_k})\stackrel{\mathrm{law}}{=}(u^{\eps_k},w)$, we have 
    \begin{equs}
        \|\widehat{\mathbb{K}}^{\eps_k,3}_{s,t}\|_{L_2}= \|K^{\eps_k}_{s,t}\big[ \phi  g'g(P_{\cdot-s}u^{\eps_k}_s)\, \big(P_{\cdot-s}J^{\eps_k}_{0,s}[1]\big) \big]\|_{L_2}^2. 
    \end{equs}
   Next, by  Lemma \ref{lem:basic_estimate_K} with $\alpha=1/2$, we have 
    \begin{equs}
        \|K^\eps_{s,t}\big[ \phi  g'g(P_{\cdot-s}u^\eps_s)\, &\big(P_{\cdot-s}J^\eps_{0,s}[1]\big) \big]\|_{L_2}^2
        \\
        &\lesssim \eps^{1/2} \int_s^t \| \phi  g'g(P_{r-s}u^\eps_s)\, \big(P_{r-s}J^\eps_{0,s}[1]\big)\|^2_{\mathcal{C}^{1/2}_2}\, dr 
        \\
        &\lesssim \eps^{1/2} \int_s^t \| \phi\|_{C^{1/2}}^2 \| g'g(P_{r-s}u^\eps_s)\|_{\mathcal{C}^{1/2}_4} \|P_{r-s}J^\eps_{0,s}[1]\|_{\mathcal{C}^{1/2}_4}^2\, dr 
        \\
          &\lesssim \eps^{1/2} \int_s^t (1+ \|P_{r-s}u^\eps_s\|_{\mathcal{C}^{1/2}_4}^2) \|P_{r-s}J^\eps_{0,s}[1]\|_{\mathcal{C}^{1/2}_4}^2\, dr
          \\
            &\stackrel{\eqref{eq:HK_omega_holder}}{\lesssim} \eps^{1/2} \int_s^t (1+ (r-s)^{-1/4}\|u^\eps_s\|_{\mathcal{C}^{1/4}_4}^2)  (r-s)^{-1/4}\|J^\eps_{0,s}[1]\|_{\mathcal{C}^{1/4}_4}^2\, dr.
    \end{equs}
    By Theorem \ref{lem:niform_bound_1/4} we have $\|u^\eps_s\|_{\mathcal{C}^{1/4}_4}\lesssim 1 $ while by \eqref{eq:J-C1/4} (applied with $f\equiv1$) we also have $\|J^\eps_{0,s}[1]\|_{\mathcal{C}^{1/4}_4}\lesssim 1 $. Consequently, 
    \begin{equs}
         \|K^\eps_{s,t}\big[ \phi  g'g(P_{\cdot-s}u^\eps_s)\, &\big(P_{\cdot-s}J^\eps_{0,s}[1]\big) \big]\|_{L_2}^2 \lesssim \eps^{1/2},
    \end{equs}
   from which the claim follows.  The case $m=2$ follows in a similar but easier manner since 
   \begin{equs}
         \|\widehat{\mathbb{K}}^{\eps_k,2}_{s,t}\|_{L_2}= \|K^{\eps_k}_{s,t}[\phi \, g(P_{\cdot-s}u^\eps_s)]\|_{L_2}.
   \end{equs}

\end{proof}
Next, we want an estimate for the term 
\begin{equs}
   \|\widehat{\mathbb{K}}^{\eps_k,4}_{s,t}\|_{L_2}  =\big\| K^{\eps_k}_{s,t}\big[ \phi  g'(P_{\cdot-s}u^{\eps_k}_s)\big(J^{\eps_k}_{s,\cdot}[g(P_{\diamond-s}u^{\eps_k}_s)] -g(P_{\cdot-s}u^{\eps_k}_s)J^{\eps_k}_{s,\cdot}[1]\big)   \big]\big\|_{L_2}
\end{equs}
For this, we first derive an estimate for $J^\eps_{s,r}[g(P_{\diamond-s}u^\eps_s)] -g(P_{r-s}u^\eps_s)J^\eps_{s,r}[1]$, which we further decompose as
\begin{equs}
    J^\eps_{s,r}[g(P_{\diamond-s}u^\eps_s)] -g(P_{r-s}u^\eps_s)J^\eps_{s,r}[1]&= J^\eps_{s,r}[g(P_{r-s}u^\eps_s)] -g(P_{r-s}u^\eps_s)J^\eps_{s,r}[1]
    \\
    & \qquad + J^\eps_{s,r}[g(P_{\diamond-s}u^\eps_s)-g(P_{r-s}u^\eps_s)].\label{eq:decomposition123}
\end{equs}
\begin{lemma} \label{lem:J_frozen_space}
    There exists a constant $N=N(p, C_g, \|\psi\|_{C^{1/4}})$ such that for all $\eps>0$ and $(s,r) \in [0,1]^2_{\leq}$ we have 
    \begin{equs}
        \|J^\eps_{s,r}[g(P_{r-s}u^\eps_s)] -g(P_{r-s}u^\eps_s)J^\eps_{s,r}[1]\|_{\mathcal{C}^{1/4}_p} \leq N |r-s|^{1/8}.\label{eq:aaa}
    \end{equs}
\end{lemma}
    \begin{proof}
   Let us fix $(s,r) \in [0,1]^2_{\leq}$.  To ease the notation, let us set $f:=g(P_{r-s}u^\eps_s)$.  For arbitrary $x,x' \in \mathbb{T}$, we have 
\begin{equs}
   J^\eps_{s,r}&[f](x) -f(x) J^\eps_{s,r}[1](x)-J^\eps_{s,r}[f](x') +f(x')J^\eps_{s,r}[1](x')
  \\&= \eps^{3/4}\int_s^r \int_\mathbb{T} F_{r,\tau}(x,x',y) \, \nabla \xi_\eps(dy, d\tau)
  \\
&=\eps^{3/4}\int_s^r\int_\T\int_\T F_{r,\tau}(x,x',y) \nabla p_{\eps^2}(y-z) \xi(dz, d\tau),
\end{equs}
where 
\begin{equs}
    F_{s,r,\tau}(x,x',y):= p_{r-\tau}(y-x)\big(f(y)-f(x)\big) -p_{r-\tau}(y-x')\big(f(y)-f(x')\big).
\end{equs}
By the fundamental theorem of calculus
\begin{equ}
    f(y)-f(x)=(x-y)
    D_f(y,x):=(x-y)  \int_0^1 \nabla f (\theta y+(1-\theta)x) \, d \theta.
\end{equ}
Therefore, using the fact that $zp_{t} (z) =t\nabla p_{t} (z)$, we see that 
 \begin{equs}
      F_{s,r,\tau}(x,x',y)& = (r-\tau) \nabla p_{r-\tau}(y-x) D_f(y,x)-(r-\tau) \nabla p_{r-\tau}(y-x') D_f(y,x').
 \end{equs}
 Consider a point $z\in\T$ and further decompose as
 \begin{equs}
     F_{r,\tau}(x,x',y)&  =(r-\tau) \Big(\big( \nabla p_{r-\tau}(y-x)-\nabla p_{r-\tau}(y-x')\Big)  D_f(z,x)
     \\
     & \qquad \qquad+  \big( \nabla p_{r-\tau}(y-x)-\nabla p_{r-\tau}(y-x')\Big)  \big(D_f(y,x)-D_f(z,x) \big) 
      \\
     & \qquad\qquad- \nabla p_{r-\tau}(y-x') \big(D_f(z,x')-D_f(z,x)\big)
     \\
     &\qquad\qquad-  \nabla p_{r-\tau}(y-x') \big(D_f(y,x')-D_f(y,x)-D_f(z,x')+D_f(z,x)\big)\Big)
     \\
     &=: (r-\tau)\sum_{i=1}^4F^{i}_{s,r,\tau}(x,x',y,z).
 \end{equs}
 Consequently, by the BDG and Minkowski's inequality, we have 
 \begin{equs}
     \| J^\eps_{s,r}&[f](x) -f (x)J^\eps_{s,r}[1](x)-J^\eps_{s,r}[f](x') +f(x')J^\eps_{s,r}[1](x')\|_{L_p}^2 
     \\
    & \lesssim \eps^{3/2} \sum_{i=1}^4 \int_s^r (r-\tau)^2\Big\| \int_{\mathbb{T}} \Big( \int_\mathbb{T}F^{i}_{s,r,\tau}(x,x',y,z) \nabla p_{\eps^2}(y-z) \, dy\Big)^2 \, dz d\tau \Big\|_{L_{p/2}}
    \\
    & =: \eps^{3/2} \sum_{i=1}^4\int_s^r(r-\tau)^2 C_i(s,\tau,r,x,x')\,d\tau  .\label{eq:bound_by_Cis}
 \end{equs}
 To prove the claim, we are therefore left to show
 \begin{equ}\label{eq:goal-C}
     C_i(\tau,r,x,x')\lesssim (r-\tau)^{-11/4}\eps^{-3/2}|x-x'|^{1/2}
 \end{equ}
 for each $i=1,2,3,4$. Indeed, using \eqref{eq:goal-C} in \eqref{eq:bound_by_Cis}, integrating in $\tau$, taking square root, dividing by $|x-x'|^{1/4}$, and taking supremum over $x\neq x'\in\T$, we arrive at the claimed bound \eqref{eq:aaa}.

 For $C_1$, first note that by integration by parts and the semigroup property,  we have
 \begin{equs}
 \int_\mathbb{T}&F^{1}_{r,\tau}(x,x',y,z) \nabla p_{\eps^2}(y-z) \, dy    
     =\big(\Delta p_{\eps^2+r-\tau}p(z-x)- \Delta p_{\eps^2+r-\tau}p(z-x')\big)  D_f(z,x).
 \end{equs}
 By this and Minkowski's inequality, it follows that 
 \begin{equs}
     C_1(s,\tau,r,x,x') & \lesssim  \int_{\mathbb{T}} \big(\Delta p_{\eps^2+r-\tau}p(z-x)- \Delta p_{\eps^2+r-\tau}p(z-x')\big)^2  \|D_f(z,x)\|_{L_p}^2\, dz 
     \\
     & \lesssim \|\Delta p_{\eps^2+r-\tau}p(\cdot-x)- \Delta p_{\eps^2+r-\tau}p(\cdot-x')\|^2_{L_2(\mathbb{T})}  \|\nabla f\|_{\mathcal{C}^0_p}^2
     \\ & \stackrel{\eqref{eq:HK_alpha}}{\lesssim} |x-x'|^{1/2}  (\eps^2+r-\tau)^{-11/4} \|\nabla f\|_{\mathcal{C}^0_p}^2.
 \end{equs}
  Further, since $g$ has bounded derivative,  we have that 
 \begin{equs}
     \|\nabla f\|_{\mathcal{C}^0_p}^2=\|\nabla \big( g(P_{r-s}u^\eps)\big)\|_{\mathcal{C}^0_p}^2\lesssim \|\nabla P_{r-s}u^\eps_s\|_{\mathcal{C}^0_p}^2 \stackrel{\eqref{eq:HK_omega_deriv}}{\lesssim} (r-s)^{-3/4} \|u^\eps_s\|_{\mathcal{C}^{1/4}_p}^2\lesssim (r-s)^{-3/4},\label{eq:nablafC0}
 \end{equs}
 where we have used Theorem \ref{lem:niform_bound_1/4} in the last inequality.
 The above two bounds, together with the elementary inequality
   $ (\eps^2+r-\tau)^{-11/4}\leq \eps^{-3/2} (r-\tau)^{-2}$ yield \eqref{eq:goal-C} for $i=1$.
 
  We proceed with $C_2$.  By the fact that $\nabla p_{\eps^2}(y-z)= \eps^{-2} (y-z)p^{1/2}_{\eps^2}(y-z)p_{\eps^2}^{1/2}(y-z)$ and the Cauchy-Schwarz inequality, we have  
   \begin{equs}
   \,      \Big( &\int_\mathbb{T}F^{2}_{r,\tau}(x,x',y,z) \nabla p_{\eps^2}(y-z) \, dy\Big)^2
        \\
        &=  \, \eps^{-4} \Big( \int_\mathbb{T}  \big( \nabla p_{r-\tau}(y-x)-\nabla p_{r-\tau}(y-x')\big)p^{1/2}_{\eps^2}(y-z) 
        \\
        &\qquad\qquad\qquad\times \big(D_f(y,x)-D_f(z,x) \big) (y-z)p^{1/2}_{\eps^2}(y-z) \, dy\Big)^2
        \\
        &\leq  \,  \eps^{-4}P_{\eps^2}| \nabla p_{r-\tau}(\cdot-x)-\nabla p_{r-\tau}(\cdot-x')|^2(z) \int_{\mathbb{T}}|D_f(y,x)-D_f(z,x)|^2|y-z|^2p_{\eps^2}(y-z) \, dy.
   \end{equs}
   Hence, by Minkowski's inequality we get 
   \begin{equs}
       C_2(s,\tau,r,x,x') & \leq \eps^{-4} \int_{\mathbb{T}} P_{\eps^2}| \nabla p_{r-\tau}(\cdot-x)-\nabla p_{r-\tau}(\cdot-x')|^2(z)
       \\
       & \qquad\qquad \times \int_{\mathbb{T}}\|D_f(y,x)-D_f(z,x)\|_{L_p}^2|y-z|^2p_{\eps^2}(y-z) \, dy  dz 
       \\
        & \lesssim \eps^{-4}  \int_{\mathbb{T}} P_{\eps^2}| \nabla p_{r-\tau}(\cdot-x)-\nabla p_{r-\tau}(\cdot-x')|^2(z)
       \\
       & \qquad\qquad \times [\nabla f]_{\mathcal{C}^{1/4}_p}^2\int_{\mathbb{T}}|y-z|^{5/2}p_{\eps^2}(y-z) \, dy  dz
       \\
          & \stackrel{\eqref{eq:HK-power}}{\lesssim} \eps^{-3/2}[\nabla f]_{\mathcal{C}^{1/4}_p}^2  \| \nabla p_{r-\tau}(\cdot-x)-\nabla p_{r-\tau}(\cdot-x')\|_{L_2(\mathbb{T})}^2 
           \\
           & \stackrel{\eqref{eq:HK_alpha}}{\lesssim}  \eps^{-3/2}|x-x'|^{1/2}[\nabla f]_{\mathcal{C}^{1/4}_p}^2 |r-s|^{-7/4}\label{eq:C2_first}.
   \end{equs}
    Then, since $g$ has bounded first and second derivatives, we have
   \begin{equs}
      \,  [\nabla f]_{\mathcal{C}^{1/4}_p}=  [\nabla \big(g(P_{r-s}u^\eps_s)\big)]_{\mathcal{C}^{1/4}_p} &\lesssim  [\nabla P_{r-s}u^\eps_s]_{\mathcal{C}^{1/4}_p}+  \|\nabla P_{r-s}u^\eps_s\|_{\mathcal{C}^{0}_p} [P_{r-s}u^\eps_s]_{\mathcal{C}^{1/4}_p}
      \\
      & \stackrel{\eqref{eq:HK_omega_deriv}, \eqref{eq:HK_omega_holder}}{\lesssim}  [\nabla P_{r-s}u^\eps_s]_{\mathcal{C}^{1/4}_p}+ |r-s|^{-3/8}\|u^\eps_s\|_{\mathcal{C}^{1/4}_p}^2,
   \end{equs}
   Moreover, we have 
   \begin{equs}
      \,   [\nabla P_{r-s}u^\eps_s]_{\mathcal{C}^{1/4}_p}=  [ P_{\frac{r-s}{2}}\nabla  P_{\frac{r-s}{2}}u^\eps_s]_{\mathcal{C}^{1/4}_p}\stackrel{\eqref{eq:HK_omega_holder}}{\lesssim} (r-s)^{-1/8}\|\nabla  P_{\frac{r-s}{2}}u^\eps_s\|_{\mathcal{C}^{0}_p}\stackrel{\eqref{eq:HK_omega_deriv}}{\lesssim} (r-s)^{-1/2}\|u^\eps_s\|_{\mathcal{C}^{1/4}_p}.
   \end{equs}
   Combining this with the above and using Theorem \ref{lem:niform_bound_1/4}, we obtain that 
   \begin{equs} \label{eq:est_nabla_f-quarter}
        [\nabla f]_{\mathcal{C}^{1/4}_p}\lesssim |r-s|^{-1/2},
   \end{equs} 
   and inserting this in \eqref{eq:C2_first}, we obtain the desired bound \eqref{eq:goal-C} for $i=2$.
   
   Moving to $C_3$, by integration by parts and the semigroup property, we have 
   \begin{equs}
   \int_\mathbb{T}F^{3}_{r,\tau}(x,x',y,z) \nabla p_{\eps^2}(y-z) \, dy
   =  \big(\Delta p_{\eps^2+r-\tau}(z-x')\big) (D_f(z,x')-D_f(z,x)).
 \end{equs}
Hence, by Minkowski's inequality we get 
\begin{equs}
    C_3(s,\tau,r,x,x') & \lesssim  \int_{\mathbb{T}} |\Delta p_{\eps^2+r-\tau}(z-x')|^2 \|D_f(z,x')-D_f(z,x)\|_{L_p}^2\, dz 
    \\
    & \lesssim |x-x'|^{1/2}\|\Delta p_{\eps^2+r-\tau}\|_{L_2(\mathbb{T})}^2 [\nabla f]_{\mathcal{C}^{1/4}_p}^2 
    \\
    & \stackrel{\eqref{eq:HK_0}}{\lesssim}  |x-x'|^{1/2}(\eps^2+r-\tau)^{-5/2}[\nabla f]_{\mathcal{C}^{1/4}_p}^2 .
\end{equs}
Using \eqref{eq:est_nabla_f-quarter} and the elementary inequality $(\eps^2+r-\tau)^{-5/2}\leq \eps^{-3/2}(r-\tau)^{-7/4}$, we get the desired bound \eqref{eq:goal-C} for $i=3$.

Finally, we move to $C_4$.  By the fact that $\nabla p_{\eps^2}(y-z)= \eps^{-2} (y-z)p^{1/2}_{\eps^2}(y-z)p_{\eps^2}^{1/2}(y-z)$ and the Cauchy-Schwarz inequality, we have  
    \begin{equs}
   \,       \Big( &\int_\mathbb{T}F^{4}_{r,\tau}(x,x',y,z) \nabla p_{\eps^2}(y-z) \, dy\Big)^2
        \\
        &=  \,  \eps^{-4}\Big( \int_\mathbb{T}  \nabla p_{r-\tau}(y-x') p_{\eps^2}^{1/2}(y-z)
        \\
        &\qquad\qquad\big(D_f(y,x')-D_f(y,x)-D_f(z,x')+D_f(z,x) \big) (y-z)p_{\eps^2}^{1/2}(y-z) \, dy\Big)^2
        \\
        &\lesssim  \,  \eps^{-4} P_{\eps^2}| \nabla p_{r-\tau}(\cdot-x')|^2(z) 
        \\
        & \qquad \times \int_{\mathbb{T}} \big(D_f(y,x')-D_f(y,x)-D_f(z,x')+D_f(z,x)\big)^2|y-z|^2p_{\eps^2}(y-z) \, dy.
   \end{equs}
   Hence, by Minkowski's inequality we get 
   \begin{equs}
       C_4(s,\tau,r,x,x') &\lesssim  \eps^{-4} \int_{\mathbb{T}} P_{\eps^2}| \nabla p_{r-\tau}(\cdot-x')|^2(z) \label{eq:est_C4_1}
        \\
        & \quad \times \int_{\mathbb{T}} \|D_f(y,x')-D_f(y,x)-D_f(z,x')+D_f(z,x)\|_{L_p}^2|y-z|^2p_{\eps^2}(y-z) \, dy dz  .      
   \end{equs}
   To estimate the four point difference, first note that trivially
\begin{equs}
          \delta_f(x,x',y,z):=\|D_f(y,x')-D_f(y,x)-D_f(z,x')+D_f(z,x)\|_{L_p}\lesssim \|\nabla f\|_{\mathcal{C}^0_p}.\label{eq:4-point-triv}
     \end{equs}
    Grouping the terms one way, one gets
    \begin{equ}
        \delta_f(x,x',y,z)\leq
        \|D_f(y,x')-D_f(y,x)\|_{L_p}+\|D_f(z,x')+D_f(z,x)\|_{L_p}\lesssim\|\nabla^2 f\|_{\mathcal{C}^0_p}|x-x'|.\label{eq:4-point-xx'}
    \end{equ}
    Similarly, but with a different grouping of terms, one has
\begin{equ}
        \delta_f(x,x',y,z)\leq
        \|D_f(y,x')-D_f(z,x')\|_{L_p}+\|D_f(y,x)+D_f(z,x)\|_{L_p}\lesssim\|\nabla^2 f\|_{\mathcal{C}^0_p}|y-z|.\label{eq:4-point-yz}
    \end{equ}
Raising \eqref{eq:4-point-triv} to the power $1/2$ and both \eqref{eq:4-point-xx'} and \eqref{eq:4-point-yz} to the power $1/4$ we get
\begin{equ}\label{eq:est_4_point_nabla3_nabla}
     \delta_f(x,x',y,z)\lesssim |x-x'|^{1/4}|y-z|^{1/4}\|\nabla f\|_{\mathcal{C}^0_p}^{1/2} \|\nabla^2 f\|_{\mathcal{C}^0_p}^{1/2}.
\end{equ}
The norm $\|\nabla f\|_{\mathcal{C}^0_p}$ is already bounded above in \eqref{eq:nablafC0}.
To bound $\|\nabla^2 f\|_{\mathcal{C}^0_p}$, we proceed similarly, this time using two bounded derivatives from $g$:
  \begin{equs}
      \| \nabla^2 f\|_{\mathcal{C}^0_p}  =   \| \nabla^2 \big(g(P_{r-s}u^\eps_s))\|_{\mathcal{C}^0_p}
            & \lesssim  \| \nabla P_{r-s}u^\eps_s\|_{\mathcal{C}^0_{2p}}^2+ \| \nabla^2 P_{r-s}u^\eps_s\|_{\mathcal{C}^0_p}
      \\
      & \stackrel{\eqref{eq:HK_omega_deriv}}{\lesssim} (r-s)^{-11/8}\|u^\eps_s\|_{\mathcal{C}^{1/4}_p}\lesssim (r-s)^{-11/8}, \label{eq:est_nabla_3_f}
  \end{equs}
  where we have used Theorem \ref{lem:niform_bound_1/4} for the last inequality.
  Inserting \eqref{eq:nablafC0} and \eqref{eq:est_nabla_3_f} in \eqref{eq:est_4_point_nabla3_nabla}, we get 
  \begin{equs}
      \delta_f(x,x',y,z)\lesssim |x'-x|^{1/4}|y-z|^{1/4} (r-s)^{-5/8}.
  \end{equs}
  By replacing the above in \eqref{eq:est_C4_1}, we get 
  \begin{equs}
      C_4(s,\tau,r,x,x')&\lesssim   \eps^{-4} |x'-x|^{1/2}(r-s)^{-5/4} \int_{\mathbb{T}}  P_{\eps^2}| \nabla p_{r-\tau}(\cdot-x')|^2(z) 
         \int_{\mathbb{T}}|y-z|^{5/2} p_{\eps^2}(y-z) \, dy dz 
        \\
        & \stackrel{\eqref{eq:HK-power}}{\lesssim}  \eps^{-3/2} |x'-x|^{1/2}(r-s)^{-5/4} \| \nabla p_{r-\tau}\|_{L_2(\mathbb{T})}^2 
        \\
        & \stackrel{\eqref{eq:HK_0}}{\lesssim}  \eps^{-3/2} |x'-x|^{1/2} (r-s)^{-5/4}(r-\tau)^{-3/2}.
  \end{equs}
 This finishes the proof. 
\end{proof}

 \begin{lemma}       \label{lem:estimate_J_frozen_space_time}
    There exists a constant $N=N(p, C_g, \|\psi\|_{C^{1/4}})$ such that for all $\eps>0$ and $(s,r) \in [0,1]^2_{\leq}$ we have 
    \begin{equs}
        \|J^\eps_{s,r}[g(P_{\diamond-s}u^\eps_s)-g(P_{r-s}u^\eps_s)]\|_{\mathcal{C}^{1/4}_p}\leq N |r-s|^{1/8}.
    \end{equs}
\end{lemma}
\begin{proof}
    By \eqref{eq:J-C1/4}, we have 
    \begin{equs}
      \,     \|J^\eps_{s,r}[g(P_{\diamond-s}u^\eps_s) & -g(P_{r-s}u^\eps_s)]\|_{\mathcal{C}^{1/4}_p}^2
          \\
         &\lesssim  \,  \,  \|g(P_{\diamond-s}u^\eps_s)-g(P_{r-s}u^\eps_s)\|_{\mathcal{C}^{0,0}_p([s,r])}^2 
         \\
         & \quad + \int_s^r (r-\tau)^{-3/4}[g(P_{\diamond-s}u^\eps_s)-g(P_{r-s}u^\eps_s)]_{\mathcal{C}^{0,1/4}_p([s,\tau])}^2\, d\tau. \label{eq:bad_power_good_power}
    \end{equs}
    For the second summand on the right hand side, we have for any $s'\in[s,\tau]$
    \begin{equs}
      \,   [g(P_{s'-s}u^\eps_s)-g(P_{r-s}u^\eps_s)]_{\mathcal{C}^{1/4}_p}& \leq  \|g(P_{s'-s}u^\eps_s)\|_{\mathcal{C}^{1/4}_p}+\|g(P_{r-s}u^\eps_s)\|_{\mathcal{C}^{1/4}_p}
      \\
      & \lesssim 1+ \|P_{s'-s}u^\eps_s\|_{\mathcal{C}^{1/4}_p}+\|P_{r-s}u^\eps_s\|_{\mathcal{C}^{1/4}_p}
      \\
      & \stackrel{\eqref{eq:HK_omega_holder}}{\lesssim} 1+\|u^\eps\|_{\mathcal{C}^{0,1/4}_p([0,1])} \lesssim 1 \label{eq:similar_estimates}
    \end{equs}
    where we have used  Theorem \ref{lem:niform_bound_1/4} for the last inequality. It follows then that
   \begin{equs}
       \int_s^r (r-\tau)^{-3/4}[g(P_{\diamond-s}u^\eps_s)-g(P_{r-s}u^\eps_s)]_{\mathcal{C}^{0,1/4}_p([s,\tau])}^2\, d\tau \lesssim |r-s|^{1/4}. \label{eq:est_good_power}
   \end{equs} 
   In addition, for any $s'\in [s,r]$ we have 
   \begin{equs}
        \|g(P_{s'-s}u^\eps_s)-g(P_{r-s}u^\eps_s)\|_{\mathcal{C}^0_p} & \lesssim  \|P_{s'-s}u^\eps_s-P_{r-s}u^\eps_s\|_{\mathcal{C}^0_p}
         = \Big\| \int_{s'-s}^{r-s} \Delta P_\theta u^\eps_s \, d \theta \Big\|_{\mathcal{C}^0_p}
        \\
        & \lesssim  \int_{s'-s}^{r-s} \|\Delta P_\theta u^\eps_s\|_{\mathcal{C}^0_p}  \, d \theta
        \stackrel{ \eqref{eq:HK_omega_deriv}}{\lesssim} (r-s')^{1/8} \|u^\eps\|_{\mathcal{C}^{0,1/4}_p([0,1])},
   \end{equs}
 Using this and Theorem \ref{lem:niform_bound_1/4}, it follows that 
   \begin{equs}
   \|g(P_{\diamond-s}u^\eps_s)-g(P_{r-s}u^\eps_s)\|_{\mathcal{C}^{0,0}_p([s,r])}^2 \lesssim |r-s|^{1/4}.
   \end{equs} 
   From this combined with \eqref{eq:est_good_power} and \eqref{eq:bad_power_good_power}, we obtain 
   \begin{equs}
       \|J^\eps_{s,r}[g(P_{\diamond-s}u^\eps_s)] -J^\eps_{s,r}[g(P_{r-s}u^\eps_s)]\|_{\mathcal{C}^{1/4}_p} \lesssim |r-s|^{1/8}.
   \end{equs}
The proof is finished.
\end{proof}

 \begin{lemma}          \label{lem:K4}
   Let $\phi \in C^{1/4}(\T)$.  There exists a constant $N=N( \|\phi\|_{C^{1/4}},C_g, \|\psi\|_{C^{1/4}})$ such that for all $k \in \mathbb{N}_1$ we have 
    \begin{equs}
        \|\widehat{\mathbb{K}}^{\eps_k,4}_{s,t}\|_{L_2}  \leq N |t-s|^{5/8}.
    \end{equs}
\end{lemma}
\begin{proof}
Since $(\widehat{u}^{\eps_k},\widehat{w}^{\eps_k})\stackrel{\mathrm{law}}{=}(u^{\eps_k},w)$, we have 
\begin{equs}
 \|\widehat{\mathbb{K}}^{\eps_k,4}_{s,t}\|_{L_2}= \| K^{\eps_k}_{s,t}\big[ \phi  g'(P_{\cdot-s}u^{\eps_k}_s)\big(J^{\eps_k}_{s,\cdot}[g(P_{\diamond-s}u^{\eps_k}_s)] -g(P_{\cdot-s}u^{\eps_k}_s)J^{\eps_k}_{s,\cdot}[1]\big)   \big]\|_{L_2} 
 \end{equs}
   For any $\eps>0$, applying Lemma \ref{lem:basic_estimate_K} with $\alpha=1/4$ yields 
    \begin{equs}
        \,\| K^{\eps}_{s,t}&\big[ \phi  g'(P_{r-s}u^\eps_s)\big(J^\eps_{s,\cdot}[g(P_{\diamond-s}u^\eps_s)] -g(P_{\cdot-s}u^\eps_s)J^\eps_{s,\cdot}[1]\big)   \big]\|_{L_2}
        \\
        \lesssim & \, \|\phi  g'(P_{\cdot-s}u^\eps_s)\big(J^\eps_{s,\cdot}[g(P_{\diamond-s}u^\eps_s)] -g(P_{\cdot-s}u^\eps_s)J^\eps_{s,\cdot}[1]\big)\|_{\mathcal{C}^{0,1/4}_2([s,t])}|t-s|^{1/2}
        \\
         \lesssim & \, 
         \| g'(P_{\cdot-s}u^\eps_s)\|_{\mathcal{C}^{0,1/4}_4([s,t])} \|J^\eps_{s,\cdot}[g(P_{\diamond-s}u^\eps_s)] -g(P_{\cdot-s}u^\eps_s)J^\eps_{s,\cdot}[1]\|_{\mathcal{C}^{0,1/4}_4([s,t])}|t-s|^{1/2}.\label{eq:K4-bigbound}
    \end{equs}
    Moreover, we have 
    \begin{equs}  \label{eq:bound_g'_14}
         \| g'(P_{\cdot-s}u^\eps_s)\|_{\mathcal{C}^{0,1/4}_4([s,t])}\lesssim 1+  \| P_{\cdot-s}u^\eps_s\|_{\mathcal{C}^{0,1/4}_4([s,t])} \lesssim 1,
    \end{equs}
    where we have used \eqref{eq:HK_omega_holder} and Theorem \ref{lem:niform_bound_1/4}. Finally, by Lemmata \ref{lem:J_frozen_space}, \ref{lem:estimate_J_frozen_space_time}, and the decomposition \eqref{eq:decomposition123}
    we have 
    \begin{equs}
        \|J^\eps_{s,\cdot}[g(P_{\diamond-s}u^\eps_s)] -g(P_{\cdot-s}u^\eps_s)J^\eps_{s,\cdot}[1]\|_{\mathcal{C}^{0,1/4}_4([s,t])}\lesssim |t-s|^{1/8}.
    \end{equs}
    Combining the above estimates proves the claim. 
\end{proof}

 \begin{lemma}              \label{lem:K5}
   Let $\phi \in C^{1/4}(\T)$.  There exists a constant $N=N( \|\phi\|_{C^{1/4}},C_g, \|\psi\|_{C^{1/4}})$ such that for all $k \in \mathbb{N}_1$ and $(s,t) \in [0,1]^2_{\leq}$ we have 
  \begin{equs}
 \|\widehat{\mathbb{K}}^{\eps_k,5}_{s,t}\|_{L_2}   \leq N|t-s|^{5/8}.
\end{equs}
\end{lemma}

\begin{proof}
We have 
 \begin{equs}
 \|\widehat{\mathbb{K}}^{\eps_k,5}_{s,t}\|_{L_2} =\|K^{\eps_k}_{s,t}\big[ \phi  g'(P_{\cdot-s}u^{\eps_k}_s)J^{\eps_k}_{s,\cdot}[g(u^{\eps_k})-g(P_{\diamond-s}u^{\eps_k}_s)] \|_{L_2}.
\end{equs}
Just as in \eqref{eq:K4-bigbound}-\eqref{eq:bound_g'_14}, we have for any $\eps>0$
\begin{equs}
        \|K^\eps_{s,t}&\big[ \phi  g'(P_{\cdot-s}u^\eps_s)J^\eps_{s,\cdot}[g(u^\eps)-g(P_{\diamond-s}u^\eps_s)] \|_{L_2} 
     \\ \lesssim &    \| J^\eps_{s,\cdot}[g(u^\eps)-g(P_{\diamond-s}u^\eps_s)]\|_{\mathcal{C}^{0,1/4}_4([s,t])}|t-s|^{1/2}  .\label{eq:whatever}
    \end{equs}
Moreover, by \eqref{eq:J-C1/4}, we have for any $r \in [s,t]$
\begin{equs}
    \| J^\eps_{s,r}[g(u^\eps)-g(P_{\diamond-s}u^\eps_s)]\|_{\mathcal{C}^{1/4}_4}^ 2 & \lesssim \|g(u^\eps)-g(P_{\diamond-s}u^\eps_s)\|^2_{\mathcal{C}^{0,0}_4([s,r])}
    \\
    & \qquad \qquad +\int_s^r (r-\tau)^{-3/4} [g(u^\eps)-g(P_{\diamond-s}u^\eps_s)]^2_{\mathcal{C}^{0,1/4}_4([s,\tau])}\, d \tau .
\end{equs}
Since $g$ is Lipschitz continuous, by Lemma \ref{lem:regularity_increments} we get 
\begin{equs}
    \|g(u^\eps)-g(P_{\diamond-s}u^\eps_s)\|_{\mathcal{C}^{0,0}_4([s,r])} \leq  \|u^\eps-P_{\diamond-s}u^\eps_s\|_{\mathcal{C}^{0,0}_4([s,r])}\lesssim |r-s|^{1/8}.
\end{equs}
Similarly to \eqref{eq:similar_estimates}, we have 
\begin{equs}
   \,  [g(u^\eps)-g(P_{\diamond-s}u^\eps_s)]_{\mathcal{C}^{0,1/4}_4([s,r])}\lesssim 1. 
\end{equs}
Consequently,
\begin{equs}
      \| J^\eps_{s,r}[g(u^\eps)-g(P_{\diamond-s}u^\eps_s)]\|_{\mathcal{C}^{1/4}_4}\lesssim |r-s|^{1/8},
\end{equs}
which in turn gives 
\begin{equs}
     \| J^\eps_{s,\cdot}[g(u^\eps)-g(P_{\diamond-s}u^\eps_s)]\|_{\mathcal{C}^{0,1/4}_4([s,t])}\lesssim |t-s|^{1/8}.
\end{equs}
This combined with \eqref{eq:whatever} finishes the proof. 
\end{proof}

 \begin{lemma}               \label{lem:K6}
   Let $\phi \in C^{1/4}(\T)$.  There exists a constant $N=N( \|\phi\|_{C^{1/4}},C_g, \|\psi\|_{C^{1/4}})$ such that for all $k \in \mathbb{N}_1$ and $(s,t) \in [0,1]^2_{\leq}$ we have 
  \begin{equs}
 \|\widehat{\mathbb{K}}^{\eps_k,6}_{s,t}\|_{L_2}   \leq N|t-s|^{5/8}.
\end{equs}
\end{lemma}
\begin{proof}
We have 
  \begin{equs}
 \|\widehat{\mathbb{K}}^{\eps_k,6}_{s,t}\|_{L_2}=\|K^{\eps_k}_{s,t}\big[ \phi \big(G_{s,\cdot}[u^{\eps_k}]-g'(P_{\cdot-s}u^{\eps_k}_s)\big)\,  J^{\eps_k}_{s,\cdot}[g(u^{\eps_k})]  \big]\|_{L_2} .
\end{equs}
    By Lemma \ref{lem:basic_estimate_K} applied with $\alpha=1/4$, for any $\eps>0$, we have 
 \begin{equs}
\,  \|K^\eps_{s,t}\big[ \phi \big(G_{s,\cdot}[u^\eps]-&g'(P_{\cdot-s}u^\eps_s)\big)\,  J^\eps_{s,\cdot}[g(u^\eps)]  \big]\|_{L_2}   
\\
& \lesssim    \| \phi \big(G_{s,\cdot}[u^\eps]-g'(P_{\cdot-s}u^\eps_s)\big)\,  J^\eps_{s,\cdot}[g(u^\eps)]\|_{\mathcal{C}^{0,1/4}_2([s,t])} |t-s|^{1/2}
\\
& \lesssim  \| \phi \|_{C^{1/4}}\|\big( G_{s,\cdot}[u^\eps]-g'(P_{\cdot-s}u^\eps_s)\big)\,  J^\eps_{s,\cdot}[g(u^\eps)]\|_{\mathcal{C}^{0,1/4}_2([s,t])}|t-s|^{1/2}.
\end{equs}
Hence, it suffices to show that
\begin{equs}              \label{eq:suff_D1_D2}
 \|\big( G_{s,\cdot}[u^\eps]-& g'(P_{\cdot-s}u^\eps_s)\big)\,   J^\eps_{s,\cdot}[g(u^\eps)]\|_{\mathcal{C}^{0,1/4}_2([s,t])} \lesssim |t-s|^{1/8}.
 \end{equs}
 We have 
 \begin{equs}
\|\big( G_{s,\cdot}[u^\eps]- g'(P_{\cdot-s}u^\eps_s)\big)\, &  J^\eps_{s,\cdot}[g(u^\eps)]\|_{\mathcal{C}^{0,1/4}_2([s,t])}
\\
& \leq   \|\big( G_{s,\cdot}[u^\eps]-g'(P_{\cdot-s}u^\eps_s)\big)\, J^\eps_{s,\cdot}[g(u^\eps)]\|_{\mathcal{C}^{0,0}_2([s,t])}
\\
 & \qquad + [\big( G_{s,\cdot}[u^\eps]-g'(P_{\cdot-s}u^\eps_s)\big)\,  J^\eps_{s,\cdot}[g(u^\eps)]]_{\mathcal{C}^{0,1/4}_2([s,t])}. 
 \\
 &:= D_1+D_2,
\end{equs}
Further, 
\begin{equs}
    D_1 \leq  \| G_{s,\cdot}[u^\eps]-g'(P_{\cdot-s}u^\eps_s)\|_{\mathcal{C}^{0,0}_4([s,t])}\| J^\eps_{s,\cdot}[g(u^\eps)]\|_{\mathcal{C}^{0,0}_4([s,t])}
\end{equs}
Recalling the definition of $G_{s,\cdot}[u^\eps]$  (see \eqref{eq:definition_of_G}) and using the fact that $g'$ is Lipschitz continuous, it is easy to see that  
\begin{equs}
    \| G_{s,\cdot}[u^\eps]-g'(P_{\cdot-s}u^\eps_s)\|_{\mathcal{C}^{0,0}_4([s,t])}\lesssim \|u^\eps-P_{\cdot-s}u^\eps_s\|_{\mathcal{C}^{0,0}_4([s,t])} \lesssim |t-s|^{1/8}, \label{eq:G_C_0}
\end{equs}
where we have used Lemma \ref{lem:regularity_increments} for the last inequality. Moreover, we have 
\begin{equs}
    \| J^\eps_{s,\cdot}[g(u^\eps)]\|_{\mathcal{C}^{0,0}_4([s,t])}=  \|u^\eps-P_{\cdot-s}u^\eps_s\|_{\mathcal{C}^{0,0}_4([s,t])} \lesssim |t-s|^{1/8} \label{eq:J_C_0}
\end{equs}
Consequently, 
\begin{equs}
    D_1 \lesssim |t-s|^{1/4}.
\end{equs}
For $D_2$, we have 
\begin{equs}
    D_2 & \leq  \|\big( G_{s,\cdot}[u^\eps]-g'(P_{\cdot-s}u^\eps_s)\|_{\mathcal{C}^{0,0}_4([s,t])}[J^\eps_{s,\cdot}[g(u^\eps)]]_{\mathcal{C}^{0,1/4}_4([s,t])}
    \\
    & \qquad + [ G_{s,\cdot}[u^\eps]-g'(P_{\cdot-s}u^\eps_s)]_{\mathcal{C}^{0,1/4}_4([s,t])}\|J^\eps_{s,\cdot}[g(u^\eps)]\|_{\mathcal{C}^{0,0}_4([s,t])}
    \\
    & \lesssim |t-s|^{1/8}\Big(  [ G_{s,\cdot}[u^\eps]-g'(P_{\cdot-s}u^\eps_s)]_{\mathcal{C}^{0,1/4}_4([s,t])}+[J^\eps_{s,\cdot}[g(u^\eps)]]_{\mathcal{C}^{0,1/4}_4([s,t])}\Big),
\end{equs}
where we used \eqref{eq:G_C_0} and \eqref{eq:J_C_0}. Moreover, since $g'$ is Lipchitz continuous, upon using \eqref{eq:HK_omega_holder}, it is easy to see that 
\begin{equs}
    \, [ G_{s,\cdot}[u^\eps]-g'(P_{\cdot-s}u^\eps_s)]_{\mathcal{C}^{0,1/4}_4([s,t])} & \leq \|G_{s,\cdot}[u^\eps]\|_{\mathcal{C}^{0,1/4}_4([s,t])}+\|g'(P_{\cdot-s}u^\eps_s)\|_{\mathcal{C}^{0,1/4}_4([s,t])}
    \\
    & \lesssim 1+ \|u^\eps\|_{\mathcal{C}^{0,1/4}_4([s,t])} \lesssim 1
\end{equs}
where we have used Theorem \ref{lem:niform_bound_1/4} for the last inequality. 
Moreover, by Lemma \ref{lem:regularity_increments}, we have 
\begin{equs}
 \,    [J^\eps_{s,\cdot}[g(u^\eps)]]_{\mathcal{C}^{0,1/4}_4([s,t])}\leq   \|J^\eps_{s,\cdot}[g(u^\eps)]\|_{\mathcal{C}^{0,1/4}_4([s,t])}=\|u^\eps-P_{\cdot-s}u^\eps_s\|_{\mathcal{C}^{0,1/4}_4([s,t])} \lesssim 1
\end{equs}
Consequently, we have 
\begin{equs}
    D_2 \lesssim |t-s|^{1/8}
\end{equs}
Combining the estimates for $D_1$ and $D_2$ shows \eqref{eq:suff_D1_D2}, and this in turn proves the claim. 
\end{proof}

\begin{lemma}   \label{lem:limit_equation}
Consider the filtered probability space $(\widehat{\Omega}, \widehat{\mathcal{F}}, \widehat{\mathbb{G}}, \widehat{\mathbb{P}})$ and the space-time white noise $\widehat{\zeta}$ given by 
$\widehat{\zeta}= \sum_{n=1}^\infty  e_n \D_t(8\sqrt[\leftroot{2}\uproot{2}4]{\pi} \widehat{\beta}^n)$.
Then, the random field $\widehat{u}$ is the unique solution of 
\begin{equs}     
    \D_t \widehat{u}= \Delta \widehat{u} +  \frac{1}{8\sqrt[\leftroot{2}\uproot{2}4]{\pi}
}g'g (\widehat{u})\widehat{\zeta}, \qquad u|_{t=0}= \psi.
\end{equs}
\end{lemma}

\begin{proof}
Fist note that $\widehat{u}$ is continuous in $(t,x)$ and $\widehat{\mathcal{P}}\otimes \mathcal{B}(\mathbb{T})$-measurable by construction, where $\widehat{\mathcal{P}}$ denotes the predictable $\sigma$-algebra associated to $\widehat{\mathbb{G}}$. Moreover, $\|\widehat{u}\|_{\mathcal{C}^{0,0}_2([0,1])}< \infty$ by \eqref{eq:limit_C_1/4}. Hence, it suffices to show  that for all $(t,x) \in [0,1] \times \mathbb{T}$, with probability one we have
  \begin{equs}
      \widehat{u}_t(x) = P_t\psi(x)+ \frac{1}{8\sqrt[\leftroot{2}\uproot{2}4]{\pi}}\int_0^t \int_{\T}p_{t-r}(x-y) g'g(\widehat{u}_r(y))  \widehat{\zeta} (dy, dr).
  \end{equs}
It is well-known that mild and (analytically) weak solutions in this setting are equivalent, hence, it suffices in fact to show that for every  $\phi \in C^\infty$, we have with probability one 
\begin{equs}
    (\widehat{u}_t, \phi)-(\psi, \phi)-\int_0^t (\widehat{u}_r, \Delta \phi) \, dr = \frac{1}{8\sqrt[\leftroot{2}\uproot{2}4]{\pi}}\int_0^t \int_{\mathbb{T}} g'g(\widehat{u}_r(y))\phi(y) \widehat{\zeta} (dy, dr),
\end{equs}
for all $t \in [0,1]$. With the notation of \eqref{eq:def_hat_M},  equality \eqref{eq:Walsh_vs_Krylov}, and using the fact that $\widehat{u}_0= \psi$, it suffices to show that for all $\phi \in C^\infty$, with probability one, for all $(s,t) \in [0,1]^2_{\leq}$ we have 
   \begin{equs}          \label{eq:M=tilde_A}
        \widehat{M}_t(\phi)-  \widehat{M}_s(\phi)=\sum_{n=1}^\infty \int_s^t \big(g'g(\widehat{u}_r) e_n, \phi\big) \, d\widehat{\beta}^n_r.
\end{equs} 
To show the above, we are going to use stochastic sewing,  Lemma \ref{lem:SSL-vanila}. 
For $(s,t) \in[0,1]^2_{\leq}$, let us set 
\begin{equs}
   A_{s,t}& : = \sum_{n=1}^\infty \int_s^t \big(g'g(P_{r-s}\widehat{u}_s) e_n, \phi\big) \, d\widehat{\beta}^n_r, \qquad   
    \tilde{\mathcal{A}}_t &:= \sum_{n=1}^\infty \int_0^t \big(g'g(\widehat{u}_r) e_n, \phi\big) \, d\widehat{\beta}^n_r.        \label{eq:defA}
\end{equs}
From now on, $A\lesssim B$ means $A \leq C B$ for some constant which does not depend on $(s,t)$. 
By It\^o's isometry, Parseval's identity, and the boundedness of $\phi$, we have 
\begin{equs}
    \| \tilde{\mathcal{A}}_t- \tilde{\mathcal{A}}_s-A_{s,t}\|_{L_2}^2 \lesssim  \int_s^t \E \|g'g(\widehat{u}_r)-g'g(P_{r-s}\widehat{u}_s)\|^2_{L_2(\mathbb{T})} \, dr.
\end{equs}
By Assumption \ref{as:g},  we have that $g'g$ is locally Lipschitz with linearly growing Lipschitz constant. Consequently, 
\begin{equs}
    \| \tilde{\mathcal{A}}_t- \tilde{\mathcal{A}}_s-A_{s,t}\|_{L_2}^2 &\lesssim  \int_s^t \big\|\big(1+|\widehat{u}_r|+|P_{r-s}\widehat{u}_s|\big)|\widehat{u}_r-P_{r-s}\widehat{u}_s| \big\|^2_{\mathcal{C}^0_2} \, dr
    \\
    & \lesssim  \int_s^t \|\widehat{u}_r-P_{r-s}\widehat{u}_s \|^2_{\mathcal{C}^0_4} \, dr
     \lesssim |t-s|^{5/4},
\end{equs}
where for the second inequality we have used H\"older's ineqality and \eqref{eq:limit_C_1/4} and for the third one we have used \eqref{eq:limit_time_reg}.  
Consequently, we have 
\begin{equs}        \label{eq:delta}
  \| \tilde{\mathcal{A}}_t- \tilde{\mathcal{A}}_s-A_{s,t}\|_{L_2}\lesssim |t-s|^{5/8}. 
\end{equs}
Moreover, by \eqref{eq:defA},  it is clear  that 
\begin{equs}              \label{eq:delta_cond}
    \widehat{\E}(\tilde{\mathcal{A}}_t- \tilde{\mathcal{A}}_s-A_{s,t} | \mathcal{G}_s)=0. 
\end{equs}
From \eqref{eq:delta}-\eqref{eq:delta_cond} it is also immediate that 
$\delta A_{s,u,t}$ satisfies \eqref{eq:SSL_con1}-\eqref{eq:SSL_con2} for some $\Gamma_1 \geq 0, \Gamma_2=0$, and  $\beta_1=5/8$, since $\delta \tilde{\mathcal{A}}_{s,u,t}=0$. Consequently, the assumptions of Lemma \ref{lem:SSL-vanila} are satisfied by $(A_{s,t})_{ (s,t) \in [0,1]^2_{\leq}}$. This means that there exist a unique process $\mathcal{A}$ satisfying \eqref{item:SSL_1}-\eqref{item:SSL_2}. On the other hand, $\tilde{\mathcal{A}}$ clearly satisfies \eqref{item:SSL_1} and by \eqref{eq:delta}-\eqref{eq:delta_cond} it also follows that it satisfies  \eqref{item:SSL_2}. By uniqueness, it follows that $\mathcal{A}=\tilde{\mathcal{A}}$. Then, it suffices to show that the process $\widehat{M}(\phi)=(\widehat{M}_t(\phi))_{t \in [0,1]}$ satisfies \eqref{item:SSL_1}-\eqref{item:SSL_2} too. Indeed, if this is the case, then by the uniqueness of the process $\mathcal{A}$, we have $\widehat{M}(\phi)= \tilde{\mathcal{A}}$, which is exactly \eqref{eq:M=tilde_A}. 

The fact that  $\widehat{M}(\phi)$ satisfies \eqref{item:SSL_1} follows immediately by Lemma \ref{lem:M_in_Lp}. As for \eqref{item:SSL_2},
 recall that by \eqref{eq:identity_before_limit} and \eqref{eq:K_decomp_bold_Ks}, we have 
\begin{equs}
     \widehat{M}^{\eps_k}_t(\phi)-  \widehat{M}^{\eps_k}_s(\phi)- \widehat{\mathbb{K}}^{\eps_k,1}_{s,t} = \sum_{m=2}^6\widehat{\mathbb{K}}^{\eps_k,m}_{s,t}.
\end{equs}
The above, combined with Lemmata \ref{lem:K_2K_3}, \ref{lem:K4}, \ref{lem:K5}, and \ref{lem:K6}, implies that 
\begin{equs}
     \|\widehat{M}^{\eps_k}_t(\phi)-  \widehat{M}^{\eps_k}_s(\phi)- \mathbb{K}^{\eps_k,1}_{s,t}\|_{L_2} \lesssim  \eps_k^{1/4}+|t-s|^{5/8}.     
\end{equs} 
Upon letting $\mathbb{N}_1 \ni k \to \infty$ and using \eqref{eq:lim_Meps_M} and Lemma \ref{lem:germ_delta_estimate}, we obtain 
\begin{equs}
    \|\widehat{M}_t(\phi)-\widehat{M}_s(\phi)-A_{s,t} \|_{L_2} \lesssim |t-s|^{5/8}.
\end{equs}
In addition, by \eqref{eq:defA} and the fact that $\widehat{M}(\phi)$ is a $\widehat{\mathbb{G}}$-martingale, it is immediate that 
\begin{equs}
     \widehat{\E}(\widehat{M}_t(\phi)-\widehat{M}_s(\phi)-A_{s,t} | \mathcal{G}_s)=0.
\end{equs}
The last two relations show that $\widehat{M}(\phi)$ satisfies \eqref{item:SSL_2} as well.  This brings the proof to an end. 
\end{proof}

\subsection{Proof of the main theorem}

We now have all the ingredients to prove the main theorem. 

\begin{proof}[Proof of Theorem \ref{thm:main}]
    First, notice that since $C([0,1] \times \mathbb{T})$ equipped with the $\sup$-norm is separable, the weak convergence is metrizable by the  L\'evy-Prokhorov metric.  Consequently, it suffices to show that $u^{\eps_k} \to u$ weakly for any sequence $(\eps_k)_{k \in \mathbb{N}}$ with $\lim_{k\to \infty}\eps_k=0$. By Lemma \ref{lem:noise-conv2}, we have that for every such sequence, there exists a subsequence $(\eps_k)_{k \in \mathbb{N}_1}$ and $\widehat{u}$ such that $u^{\eps_k}\to \widehat{u}$ weakly as $\mathbb{N}_1 \ni k \to \infty$. Moreover, by Lemma \ref{lem:limit_equation} and the fact that weak uniqueness holds for \eqref{eq:main_equation_limit}, we have that $u\overset{\mathrm{law}}{=} \widehat{u}$. This implies  that for the original sequence $(u^{\eps_k})_{k \in \mathbb{N}}$ we have that  $u^{\eps_k} \to u$ weakly. This finishes the proof. 
\end{proof}

\begin{appendix}

\section{Auxiliary Lemmata}

The following is a consequence of \cite[Theorem 2.2]{KurtzProtter}.
\begin{lemma}                \label{lem:convergence_stochastic_integrals}
    For $n \in \mathbb{N}\cup\{0\}$, let  $B^n$ be continuous $\mathbb{F}^n$-martingales and let $G^n$ be continuous  $\mathbb{F}^n$-adapted processes. Assume that $G^n \to G^0$ and $B^n \to B^0$, uniformly on $[0, 1]$, in probability. Then,   in probability, we have
    \begin{equs}
    \sup_{t \in [0, 1]}\big|\int_0^t G^n_s \, dB^n_s - \int_0^t G^0_s \, dB^0_s\big | \to 0.    
    \end{equs}
\end{lemma}

The following is a consequence of \cite[Lemma~7.1.1]{Henry}. For the convenience of the reader we provide a short proof.
\begin{lemma}   \label{lem:Gronwal}
    Let $\gamma >-1$. Let $f, g:[0,1]\to [0, \infty)$ such that $f$ is bounded and $g$ is non-deceasing.  Assume that for all $t \in [0,1]$ we have 
    \begin{equs}
       f(t) \leq \int_0^t (t-r)^{\gamma} \sup_{s \leq r } f(s) \, dr +g(t).
    \end{equs}
    There exists a constant $N=N(\gamma)$, depending only on $\gamma$,  such that for all $t \in [0,1]$ we have 
    \begin{equs}
        \sup_{s \in [0,t]}f(s) \leq N g(t). 
    \end{equs}
\end{lemma}

\begin{proof}
    Let $\lambda>0$ to be determined later and consider $m(t)= e^{-\lambda t}f(t)$. Then we have 
    \begin{equs}
        f(t) \leq \int_0^t (t-r)^{\gamma} e^{\lambda r} \sup_{s \leq r } m(s) \, dr +g(t),
    \end{equs}
    which implies that for arbitrary $T \in [0,1]$  and for all $t \in[0,T]$ we have 
    \begin{equs}                \label{eq:to_take_sup}
        m(t) \leq \big( \sup_{s\leq T} m(s) \big) \int_0^t (t-r)^{\gamma} e^{-\lambda (t-r) }  \, dr  +g(T). 
    \end{equs}
    Further, by a change of variables we see that for any $t \in [0,1]$ we have 
    \begin{equs}
         \int_0^t (t-r)^{\gamma} e^{-\lambda (t-r) }  \, dr  \leq \lambda^{-1-\gamma} N(\gamma). 
    \end{equs}
    Consequently, the claim follows by choosing $\lambda$ sufficiently large, taking suprema over $t \in [0,T]$ in \eqref{eq:to_take_sup} and rearranging. 
  \end{proof}  

  The following is from \cite[Theorem 2.1]{SSL}. Recall the notation 
  \begin{equs}
      \delta A_{s,u,t}= A_{s,t}-A_{s,u}-A_{u,t}.
  \end{equs}

\begin{lemma}\label{lem:SSL-vanila}
Let $d \in \mathbb{N}$, $p\in[2,\infty)$.  Let
$(A_{s,t})_{(s,t)\in[S,T]_\leq^2}$ be a family of $\R^d$-valued random variables 
such that $A_{s,t}$ is $\mathcal{F}_t$-measurable for all $(s,t)\in[0,1]_{\leq}^2$.
Suppose that there exist constants $\Gamma_1,\Gamma_2 \geq 0$, $\beta_1>1/2$, and $\beta_2>1$ such that for all $(s,u,t)\in[0,1]_{\leq}^3$ the following holds:
\begin{equs}
\|\delta A_{s,t}\|_{L_p}&\leq \Gamma_1|t-s|^{\beta_1},      \label{eq:SSL_con1}
\\
\|\E(\delta A_{s,u,t}| \mathcal{F}_s)\|_{L_p}&\leq \Gamma_2|t-s|^{\beta_2}.\label{eq:SSL_con2}
\end{equs}
  Then there exists a unique (up to modifications) process $(\mathcal{A}_t)_{t\in [0,1]}$ with the following properties: 

\begin{enumerate}[(I)]
 \item \label{item:SSL_1} $\mathcal{A}_0=0$ a.s., and for each $t \in [0,1]$, $\mathcal{A}_t$ is $\mathcal{F}_t$-measurable and $\mathcal{A}_t \in L_p$.

 \item \label{item:SSL_2} There exist $K_1,K_2 \geq 0$ such that for all  $(s,t)\in[0,1]_{\leq}^2$ we have 
\begin{equs}                      
    \|\mathcal{A}_t-\mathcal{A}_s- A_{s,t}\|_{L_p} &\leq K_1|t-s|^{\beta_1} \label{eq:SSL_uni_1}
    \\
    \|\E( \mathcal{A}_t-\mathcal{A}_s- A_{s,t}| \mathcal{F}_s)\|_{L_p}& \leq K_2|t-s|^{\beta_2}               \label{eq:SSL_uni_2}
\end{equs}
\end{enumerate}
\end{lemma}

\end{appendix}

 \section*{Acknowledgment}
KD has been supported by  the Engineering \& Physical Sciences Research Council (EPSRC) [grant number EP/Y016955/1].
MG was funded by the European Union (ERC, SPDE, 101117125). Views and opinions expressed
are however those of the author(s) only and do not necessarily reflect those of the European Union
or the European Research Council Executive Agency. Neither the European Union nor the granting
authority can be held responsible for them. 

\bibliographystyle{Martin}
\bibliography{references}

\end{document}